\documentclass[review,onefignum,onetabnum]{siamonline220329}

\usepackage{lipsum}
\usepackage{amsfonts}
\usepackage{graphicx,subfig}
\usepackage{epstopdf}
\usepackage{algorithmic}
\usepackage{amssymb}
\usepackage{breqn}
\usepackage{xr}
\usepackage{subfig}
\usepackage{bm}
\usepackage{mathtools}
\usepackage{booktabs}
\usepackage{multirow}
\usepackage{soul}
\usepackage[export]{adjustbox}
\ifpdf
  \DeclareGraphicsExtensions{.eps,.pdf,.png,.jpg}
\else
  \DeclareGraphicsExtensions{.eps}
\fi

\usepackage{enumitem}
\setlist[enumerate]{leftmargin=.5in}
\setlist[itemize]{leftmargin=.5in}

\newsiamremark{remark}{Remark}
\newsiamremark{hypothesis}{Hypothesis}
\crefname{hypothesis}{Hypothesis}{Hypotheses}
\newsiamthm{claim}{Claim}
\newsiamthm{prop}{Proposition}
\headers{Data-driven mirror descent with input-convex neural networks}{H. Y. Tan, S. Mukherjee, J. Tang, and C.-B. Sch\"onlieb}

\title{Data-Driven Mirror Descent with Input-Convex Neural Networks\thanks{Submitted to the SIAM J. on Mathematics of Data Science
}}

\author{Hong Ye Tan\thanks{Department of Applied Mathematics and Theoretical Physics, University of Cambridge, UK (\email{hyt35@cam.ac.uk}, \email{sm2467@cam.ac.uk}, \email{jt814@cam.ac.uk}, \email{cbs31@cam.ac.uk}).}
\and Subhadip Mukherjee\footnotemark[2] \thanks{Department of Computer Science, University of Bath, UK (\email{sm3655@bath.ac.uk}).}
\and Junqi Tang\footnotemark[2]
\and Carola-Bibiane Sch\"onlieb\footnotemark[2]}

\usepackage{amsopn}
\usepackage{svg}

\DeclareMathOperator{\dom}{dom}

\newcommand{\R}{\mathbb{R}}
\DeclareMathOperator*{\argmin}{arg\,min}
\DeclareMathOperator*{\sign}{sign}
\newlength{\subcolumnwidth}
\newenvironment{subcolumns}[1][0.45\columnwidth]
 {\valign\bgroup\hsize=#1\setlength{\subcolumnwidth}{\hsize}\vfil##\vfil\cr}
 {\crcr\egroup}
\newcommand{\nextsubcolumn}[1][]{%
  \cr\noalign{\hfill}
  \if\relax\detokenize{#1}\relax\else\hsize=#1\setlength{\subcolumnwidth}{\hsize}\fi
}
\newcommand{\nextsubfigure}{\vfill}

\usepackage{acro}

\ifpdf
\hypersetup{
  pdftitle={Data-Driven Mirror Descent with Input-Convex Neural Networks},
  pdfauthor={Hong Ye Tan, Subhadip Mukherjee, Junqi Tang, and Carola-Bibiane Sch\"onlieb}
}
\fi

\newcommand{\billy}[1]{}
\newcommand{\sm}[1]{}
\newcommand{\hy}[1]{}
\DeclareAcronym{MD}{
  short = MD,
  long = mirror descent} 
  
\DeclareAcronym{PDHG}{
  short = PDHG,
  long = primal-dual hybrid gradient} 
  
\DeclareAcronym{SVM}{
short = SVM,
long = support vector machine} 
  
\DeclareAcronym{LMD}{
  short = LMD,
  long = learned mirror descent} 
  
\DeclareAcronym{ICNN}{
short = ICNN,
long = input-convex neural network} 

\DeclareAcronym{lsc}{
short = l.s.c.,
long = lower semi-continuous} 

\DeclareAcronym{ISTA}{
short = ISTA,
long = iterative shrinkage and thresholding algorithm}

\begin{document}

\maketitle


\begin{abstract}
\textit{Learning-to-optimize} is an emerging framework that seeks to speed up the solution of certain optimization problems by leveraging training data. Learned optimization solvers have been shown to outperform classical optimization algorithms in terms of convergence speed, especially for convex problems. Many existing data-driven optimization methods are based on parameterizing the update step and learning the optimal parameters (typically scalars) from the available data. We propose a novel functional parameterization approach for learned convex optimization solvers based on the classical \ac{MD} algorithm. Specifically, we seek to learn the optimal Bregman distance in \ac{MD} by modeling the underlying convex function using an \ac{ICNN}. The parameters of the ICNN are learned by minimizing the target objective function evaluated at the \ac{MD} iterate after a predetermined number of iterations. The inverse of the mirror map is modeled approximately using another neural network, as the exact inverse is intractable to compute. We derive convergence rate bounds for the proposed \ac{LMD} approach with an approximate inverse mirror map and perform extensive numerical evaluation on various convex problems such as image inpainting, denoising, learning a two-class \ac{SVM} classifier and a multi-class linear classifier on fixed features.
\end{abstract}



\section{Introduction}
Convex optimization problems are pivotal in many modern data science and engineering applications. These problems can generally be formulated as
\begin{equation}\label{eq:convexEqGeneral}
    \min_{x \in \mathcal{X}} \left[f(x) + g(x)\right],
\end{equation}
where $\mathcal{X}$ is a Hilbert space, and $f,g:\mathcal{X} \rightarrow \bar{\R}$ are proper, convex, and \ac{lsc} functions. In different scenarios, $f$ and $g$ have different levels of regularity such as differentiability or strong convexity. In the context of inverse problems, $f$ can be a data fidelity loss and $g$ a regularization function. 


In the past few decades, extensive research has gone into developing efficient and provably convergent optimization algorithms for finding the minimizer of a composite objective function as in \eqref{eq:convexEqGeneral}, leading to several major theoretical and algorithmic breakthroughs. For generic convex programs with first-order oracles, optimal algorithms have been proposed under different levels of regularity \cite{nesterov2013gradient,lan2012optimal,lan2018optimal}, which are able to match the complexity lower-bounds of the problem class. Although there exist algorithms that are optimal for generic problem classes, practitioners in different scientific areas usually only need to focus on a very narrow subclass, for which usually neither tight complexity lower-bounds nor optimal algorithms are known. As such, it is extremely difficult and impractical to either find tight lower-bounds or handcraft specialized optimal algorithms for every single subclass in practice.

The aim of this work is \textit{learning to optimize} convex objectives of the form \eqref{eq:convexEqGeneral} in a provable manner. Learned optimization solvers have been proposed through various methods, including reinforcement learning and unsupervised learning \cite{andrychowicz2016GDbyGD,banert2020,gregor2010learnedcoding,li2016learningtooptimize}. The goal is to minimize a fixed loss function as efficiently as possible, which can be formulated as minimizing the loss after a certain number of iterations, or minimizing the number of iterations required to attain a certain error. The common idea is to directly parameterize the update step as a neural network, taking previous iterates and gradients as arguments. These methods have been empirically shown to speed up optimization in various settings including training neural networks \cite{li2016learningtooptimize, andrychowicz2016GDbyGD}. However, many of these methods lack theoretical guarantees, and there is a lack of principled framework for integrating machine learning into existing classical algorithms. 

Banert et al. developed a theoretically grounded method in \cite{banert2020} for parameterizing such update steps using combinations of proximal steps, inspired by proximal splitting methods. By learning the appropriate coefficients, the method was able to outperform the classical \ac{PDHG} scheme \cite{cham_pock}. However, having a fixed model limits the number of learnable parameters, and therefore the extent to which the solver can be adapted to a particular problem class. Banert et al. later drifted away from the framework of learning parameters of fixed models, and instead directly modeled an appropriate update function using a deviation-based approach, allowing for a more expressive parameterization \cite{banert2021accel}.

Learned optimizers are sometimes modeled using classical methods, as the existing convergence guarantees can lead to insights on how neural networks may be incorporated with similar convergence guarantees. Even if such guarantees are not available, such as in the case of learned \ac{ISTA}, they can still lead to better results on certain problems \cite{gregor2010learnedcoding}. Conversely, Maheswaranathan et al. showed that certain learned optimizers, parameterized by recurrent neural networks, can reproduce classical methods used for accelerating optimization \cite{maheswarnathan2020learnedopti}. By using a recurrent neural network taking the gradient as an input, the authors found that the learned optimizer expresses mechanisms including momentum, gradient clipping, and adaptive learning rates.

One related idea to our problem is meta-learning, also known as ``learning to learn". This typically concerns learning based on prior experience with similar tasks, utilizing techniques such as transfer learning, to learn how similar an optimization task is to previous tasks using statistical features \cite{vanschoren2018metaL}. Our problem setting will instead be mainly concerned with convex optimization problems, as there are concrete classical results for comparison. 

Integrating machine learning models into classical algorithms can also be found notably in Plug-and-Play (PnP) algorithms. Instead of trying to learn a solver for a general class of optimization problems, PnP methods deal with the specific class of image restoration. By using proximal splitting algorithms and replacing certain proximal steps with generic denoisers, the PnP algorithms, first proposed by Venkatakrishnan et al. in 2013, were able to achieve fast and robust convergence for tomography problems \cite{venkatakrishnan2013PnP}. This method was originally only motivated in an intuitive sense, with some analysis of the theoretical properties coming years later by Chan et al. \cite{chan2016PnP}, and more recently by Ryu et al. \cite{ryu2019PnP}. Most critically, many subsequent methods of showing convergence rely on classical analysis such as monotone operator and fixed point theory, demonstrating the importance of having a classical model-based framework to build upon.

One of the main difficulties in learning to optimize is the choice of function class to learn on. Intuitively, a more constrained function class may allow for the learned method to specialize more. However, it is difficult to quantify the similarity between the geometry of different problems. Banert et al. proposed instead to use naturally or qualitatively similar function classes in \cite{banert2021accel}, including regularized inverse problems such as inpainting or denoising, which will be used in this work as well.
               
\subsection{Contributions}
We propose to learn an alternative parameterization using mirror descent (MD), which is a well-known convex optimization algorithm first introduced by Nemirovsky and Yudin \cite{nemirovsky1983MD}. Typical applications of \ac{MD} require hand-crafted mirror maps, which are limited in complexity by the requirement of a closed-form convex conjugate. We propose to replace the mirror map  in \ac{MD} with an input convex neural network (ICNN) \cite{amos2017icnn}, which has recently proved to be a powerful parameterization approach for convex functions \cite{acr_arxiv}. By modeling the mirror map in this manner, we seek to simultaneously introduce application-specific optimization routines, as well as learn the problem geometry. 

Using our new paradigm, we are able to obtain a learned optimization scheme with convergence guarantees in the form of regret bounds\sm{This sounds cryptic. Non-trivial in what sense? If you are referring to the dependence of the regret bound on $t_k$, it is better to write it explicitly.}\hy{reworded(?)}. We observe numerically that our learned mirror descent (LMD) algorithm is able to adapt to the structure of the class of optimization problems that it was trained on, and provide significant acceleration.

This paper is organized as follows. In \cref{sec:background}, we recall the \ac{MD} algorithm and the existing convergence rate bounds. \Cref{sec:main} presents our main results on convergence rate bounds with inexact mirror maps, and a proposed procedure of `learning' a mirror map. In \cref{sec:exactMDMaps}, we will show some simple examples of both \ac{MD} and its proposed learned variant \ac{LMD} in the setting where the inverse map is known exactly. \Cref{sec:experiments} deals with numerical experiments with inverse problems in imaging and linear classifier learning.


\section{Background}\label{sec:background}
In this section, we will outline the \ac{MD} method as presented by Beck and Teboulle \cite{BECK2003167}. Convergence guarantees for convex optimization methods commonly involve a Lipschitz constant with respect to the Euclidean norm. However, depending on the function, this may scale poorly with dimension. Mirror descent circumvents this by allowing for this Lipschitz constant to be taken with respect to other norms such as the $\ell^1$ norm.  This has been shown to scale better with dimension compared to methods such as projected subgradient descent on problems including online learning and tomography \cite{bubeck2014convex, orabona2015generalized, tomographyMD2001bental}. Further work has been done by Gunasekar et al., showing that \ac{MD} is equivalent to natural/geodesic gradient descent on certain Riemannian manifolds \cite{gunasekar2021mirrorlessMD}. We will continue in the simpler setting where we have a potential given by a strictly convex $\Psi$ to aid parameterization, but this can be replaced by a suitable Hessian metric tensor. 

Let $\mathcal{X}\subset \R^n$ be a closed convex set with nonempty interior. Let $(\R^n)^*$ denote the corresponding dual space of $\R^n$.

\begin{definition}[Mirror Map]
We say $\Psi:\mathcal{X} \rightarrow \R$ is a \textbf{mirror potential} if it is continuously differentiable and strongly convex. We call the gradient $\nabla \Psi:\mathcal{X} \rightarrow (\R^n)^*$ a \textbf{mirror map}.\sm{Mention that $^*$ indicates the dual space.} \hy{added}
\end{definition}
\begin{remark}
A mirror potential $\Psi$ may also be referred to as a \textit{distance generating function}, as a convex map induces a Bregman distance $B_\Psi(x,y)$, defined by $B_\Psi(x,y) = \Psi(x) - \Psi(y) - \langle \nabla \Psi(y), x-y \rangle$. For example, taking $\Psi(x) = \|x\|_2^2$ recovers the usual squared Euclidean distance $B_\Psi(x,y) = \|x-y\|_2^2$.
\end{remark}
If $\Psi$ is a mirror potential, then the convex conjugate $\Psi^*$ defined as 
\begin{equation*}
    \Psi^*(x^*) = \sup_{x \in \mathcal{X}}\,\{\langle x^*, x\rangle - \Psi(x)\}
\end{equation*}
is differentiable everywhere, and additionally satisfies $\nabla \Psi^* = (\nabla \Psi)^{-1}$ \cite{BECK2003167,RockWets1998VA}. The (forward) mirror map $\nabla \Psi$ \textit{mirrors} from the primal space $\mathcal{X}$ into a subset of the dual space $(\R^n)^*$, and the inverse (backward) mirror map $\nabla \Psi^*$ \textit{mirrors} from the dual space $\dom\left(\nabla \Psi^*\right) \subseteq (\R^n)^*$ back into the primal space $\mathcal{X}$.

Suppose first that we are trying to minimize a convex differentiable function $f$ over the entire space $\mathcal{X} = \R^n$, $\min_{x \in \mathcal{X}} f(x)$. Suppose further for simplicity that $\dom \left(\nabla \Psi^*\right) = (\R^n)^*$. For an initial point $x_0 \in \mathcal{X}$ and a sequence of step-sizes $(t_k)_{k\ge 0},\, t_k>0$, the mirror descent iterations can be written as follows:%
\begin{equation}\label{eq:MDDef}
        y_{k} = \nabla \Psi (x_k) - t_k \nabla f (x_k),\ x_{k+1} = \nabla \Psi^* (y_k).
\end{equation}
There are two main sequences, $(x_k)_{k=0}^\infty$ in the primal space $\mathcal{X}$ and $(y_k)_{k=0}^\infty$ in the dual space $(\R^n)^*$. The gradient step at each iteration is performed in the dual space, with the mirror map $\nabla \Psi$ mapping between them. Observe that if $\Psi = \frac{1}{2}\|x\|_2^2$, then $\nabla \Psi$ is the identity map $\R^n \rightarrow (\R^n)^*$ and we recover the standard gradient descent algorithm. An equivalent formulation of the \ac{MD} update rule in \eqref{eq:MDDef} is the subgradient algorithm \cite{BECK2003167}:%
\begin{equation}\label{eq:SANP}
    x_{k+1} = \argmin_{x \in \mathcal{X}} \left\{ \langle x, \nabla f(x_k) \rangle + \frac{1}{t_k}B_{\Psi}(x, x_k)\right\}.
\end{equation}
This can be derived by using the definitions of the Bregman distance and of the convex conjugate $\Psi^*$. The convexity of $\Psi$ implies that the induced Bregman divergence $B_\Psi$ is non-negative, which allows for this iteration to be defined. Observe again that if $\Psi = \frac{1}{2}\|x\|_2^2$, then $B_\Psi(x,y) = \frac{1}{2} \|x-y\|_2^2$ and we recover the argmin formulation of the gradient descent update rule.

\ac{MD} enjoys the following convergence rate guarantees. Let $\|\cdot \|$ be a norm on $\R^n$, and $\|\cdot\|_* = \max \{\langle \cdot, x\rangle : x \in \R^n,\ \|x\|\le 1\}$ be the corresponding dual norm. For a set $\mathcal{X} \subseteq \R^n$, let $\text{int}(\mathcal{X})$ denote the interior of $\mathcal{X}$.
\sm{Before stating the theorems, write a short paragraph on notations, defining notations for interiors, dual space, etc.}\hy{added}
\begin{theorem} \cite[Thm 4.1]{BECK2003167}\label{thm:MDConvRate}
Let $\mathcal{X}$ be a closed convex subset of $\R^n$ with nonempty interior, and $f:\mathcal{X} \rightarrow \R$ a convex function. Suppose that $\Psi$ is a $\sigma$-strongly convex mirror potential. Suppose further that the following hold:
\begin{enumerate}
    \item $f$ is Lipschitz with Lipschitz constant $L_f$ with respect to $\|\cdot\|$\sm{Clearly define the notation for the underlying norm. You are using its dual in the statement.}\hy{added};
    \item The set of minimizers $\min_{x \in \mathcal{X}} f(x)$ is nonempty; let $x^*$ be a minimizer of $f$. \sm{$x^*$ is not defined before.} \hy{made clearer}
\end{enumerate}
Let $\{x_k\}_{k=1}^\infty$ be the sequence generated by the MD iterations \eqref{eq:MDDef} with starting point $x_1 \in \text{int}(\mathcal{X})$. Then the iterates satisfy the following regret bound:
\begin{equation}
    \sum_{k=1}^s t_k(f(x_k) - f(x^*))\le B_\Psi(x^*, x_1) - B_\Psi(x^*, x_{s+1}) + (2\sigma)^{-1} \sum_{k=1}^s t_k^2 \|\nabla f(x_k)\|_*^2.
\end{equation}
In particular, we have
\begin{equation}
    \min_{1\le k \le s} f(x_k) - f(x^*) \le \frac{B_\Psi(x^*, x_1) + (2\sigma)^{-1} \sum_{k=1}^s t_k^2 \|\nabla f(x_k)\|_*^2}{\sum_{k=1}^s t_k}.
\end{equation}
\end{theorem}

\begin{remark}
The proof of \cref{thm:MDConvRate} depends only on the property that $\nabla \Psi^* = (\nabla \Psi)^{-1}$. Therefore, the inverse mirror map $(\nabla \Psi^*)$ as in \cref{eq:MDDef} can be replaced with $(\nabla \Psi)^{-1}$, yielding a formulation of \ac{MD} that does not reference the convex conjugate $\Psi^*$ of the mirror potential $\Psi$ itself, but only the gradient $\nabla \Psi^*$.
\end{remark}

To motivate our goal of learning mirror maps, we will demonstrate an application of MD that drastically speeds up convergence over gradient descent. We consider optimization on the simplex $\Delta_d = \{x \in \R^d : x\ge 0,\ \sum_j x_j = 1\}$, equipped with a mirror potential given by the (negative log-) entropy map \cite{BECK2003167}. We have the following mirror maps, where logarithms and exponentials of vectors are to be taken component-wise:%

\begin{equation}\label{eq:entropicMD}
    \Psi(x) = \sum_j x_j \log x_j,\ \nabla \Psi(x) = 1+\log(x),\ \nabla \Psi^*(y) = \frac{\exp(y)}{\sum_j \exp(y_j)}.
\end{equation}%
\sm{Write somewhere in the `notations' paragraph that a function with a vector argument indicates applying the function component-wise, unless mentioned otherwise.} \hy{fixed}
This results in the \emph{entropic mirror descent} algorithm. It can be shown to have similar convergence rates as projected subgradient descent, with a $O(1/\sqrt{k})$ convergence rate \cite[Thm 5.1]{BECK2003167}. Given that the optimization is over a probability simplex, a natural problem class to consider is a probabilistic distance between points, given by the KL divergence. 

Minimizing the KL divergence is a convex problem on the simplex $x \in \Delta_d$. For a point $y \in \Delta_d$, the KL divergence is given as follows, where $0\log 0$ is taken to be 0 by convention:%
\begin{equation}
    \min_{x \in \Delta_d} KL(x\|y) = \sum_{i=1}^d x_i\log\left(\frac{x_i}{y_i}\right).
\end{equation}%
To demonstrate the potential of MD, we can apply the entropic MD algorithm to the problem classes of minimizing KL divergence and of minimizing least squares loss over the simplex $\Delta_d$. The function classes that we apply the entropic MD algorithm and gradient descent to are:
\[\mathcal{F}_{KL} = \left\{KL(\cdot\|y) : y \in \Delta_d\right\},\quad \mathcal{F}_{lsq} = \left\{\|\cdot - y\|_2^2 : y \in \Delta_d\right\},\]
where the functions have domain $\Delta_d$. Note that the true minimizers of a function in either of these function classes is given by the parameter $y \in \Delta_d$. 

To compare these two optimization algorithms, we optimize 500 functions from the respective function classes, which were generated by uniformly sampling $y$ on the simplex. \Cref{fig:simplex_comb} plots the evolution of the loss for the entropic MD algorithm and gradient descent for these two problem classes, applied with various step-sizes. The entropic MD algorithm gives linear convergence on the KL function class $\mathcal{F}_{KL}$, massively outperforming the gradient descent algorithm. However, entropic MD is unable to maintain this convergence rate over the least-squares function class $\mathcal{F}_{lsq}$. The difference in convergence rate demonstrates the importance of choosing a suitable mirror map for the target function class, as well as the potential of MD in accelerating convergence. This relationship between the function class and mirror maps motivates a learned approach to deriving mirror maps from data to replace classical hand-crafted mirror maps.

\begin{figure}
    \centering
    \includegraphics[width=\textwidth]{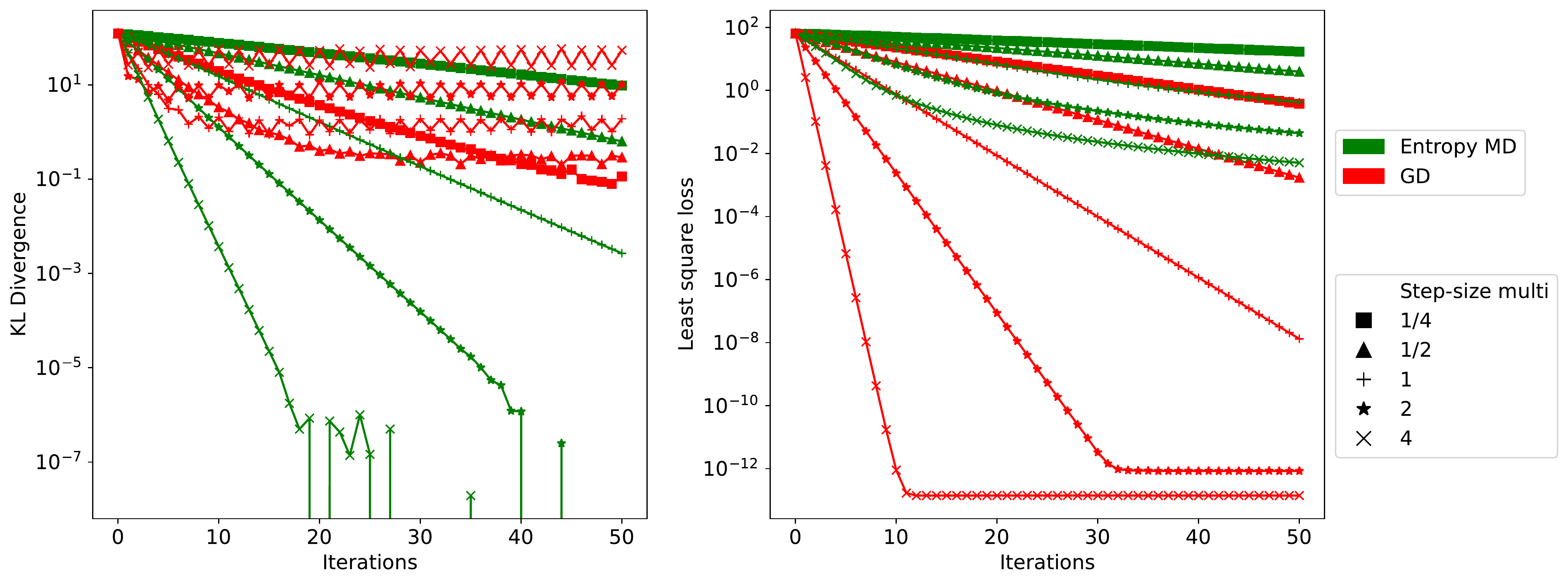}
    \caption{Effect of using the entropic MD method \eqref{eq:entropicMD} to minimize KL divergence (left) and least squares loss (right). The step-sizes were taken as $0.1\times$step-size multi. We can see that entropic MD (green) outperforms the gradient descent method for the KL divergence task, though loses out in the least squares task. The unstable iterations at low KL divergence are due to machine precision. The difference between the two optimization methods on each problem class demonstrates the potential of adapting to the optimization geometry using MD.}
    \label{fig:simplex_comb}
\end{figure}

\section{Main Results}\label{sec:main}
We first theoretically show convergence properties of mirror descent when the mirror map constraint $\nabla \Psi^* = (\nabla \Psi)^{-1}$ is only approximately satisfied. Motivated by these convergence properties, we propose our Learned Mirror Descent method, trained with a loss function balancing empirical convergence speed and theoretical convergence guarantees.

We briefly explain our key objective of approximate mirror descent. Recall that \ac{MD} as given in \eqref{eq:MDDef} requires two mirror maps, $\nabla \Psi$ and $\nabla \Psi^*$. We wish to parameterize both $\Psi$ and $\Psi^*$ using neural networks $M_\theta$ and $M_\vartheta^*$, and weakly enforce the constraint that $\nabla M_\vartheta^* = (\nabla M_\theta)^{-1}$. To maintain the convergence guarantees of \ac{MD}, we will derive a bound on the regret depending on the deviation between $\nabla M_\vartheta^*$ and $(\nabla M_\theta)^{-1}$ in a sense that will be made precise later. We will call the inconsistency between the parameterized mirror maps $\nabla M_\vartheta^*$ and $(\nabla M_\theta)^{-1}$ the \emph{forward-backward inconsistency/loss}.

Recall the problem setting as in Section \ref{sec:background}. Let $\Psi$ be a mirror potential, i.e. a $C^1$ $\sigma$-strongly-convex function with $\sigma>0$. In this section, we shall work in the unconstrained case $\mathcal{X} = \R^n$. We further assume $f$ has a minimizer $x^* \in \mathcal{X}$.

Recall the \ac{MD} iteration \eqref{eq:MDDef} with step sizes $\{t_k\}_{k=1}^\infty$ as follows. Throughout this section, $B = B_{\Psi}$ is the Bregman distance with respect to $\Psi$, and $\Psi^*$ is the convex conjugate of $\Psi$:
\begin{equation}\label{eq:MDargmin}
x_{k+1} = \arg \min_{x \in \mathcal{X}}\left\{\langle x, t_k \nabla f(x_k) \rangle + B(x, x_k)\right\} = \nabla \Psi^* (\nabla\Psi(x_k) - t_k \nabla f(x_k)).
\end{equation}%

For a general mirror map $\Psi$, the convex conjugate $\Psi^*$ and the associated backward mirror map $\nabla \Psi^*$ may not have a closed form. Suppose now that we parameterize $\Psi$ and $\Psi^*$ with neural networks $M_\theta$ and $M_\vartheta^*$ respectively, satisfying $\nabla M_\vartheta^* \approx (\nabla M_\theta)^{-1}$. The resulting \emph{approximate mirror descent} scheme is as follows, starting from $\tilde{x}_1 = x_1$:
\begin{equation}\label{eq:approximateMD}
        \tilde{x}_{k+1} = \nabla M_\vartheta^* (\nabla M_\theta(\tilde{x}_k) - t_k \nabla f(\tilde{x}_k)),\ k=1,2,\cdots.
\end{equation}
Here, we enforce that the sequence $\{\tilde{x}_{k}\}$ represents an approximation of a mirror descent iteration at each step, given by 
\begin{equation}\label{eq:trueMDiterates}
    x_{k+1} = \arg \min_{x \in \mathcal{X}}\left\{\langle x, t_k \nabla f(\tilde{x}_k) \rangle + B(x, \tilde{x}_k)\right\} = (\nabla M_\theta)^{-1} (\nabla M_\theta(\tilde{x}_k) - t_k \nabla f(\tilde{x}_k)).
\end{equation}

%

Hereafter, we will refer to $M_\theta$ and $M_\vartheta^*$ as the \emph{forward and backward (mirror) potentials}, respectively, and the corresponding gradients as the \emph{forward and backward (mirror) maps}. For practical purposes, $\{\tilde{x}_k\}$ should be considered as the iterations that we can compute. Typically, both the argmin and $\nabla \Psi^*$ are not easily computable, hence $x_k$ will not be computable either. However, defining this quantity will prove useful for our analysis, as we can additionally use this quantity to compare how close the forward and backward maps are from being inverses of each other.



The following theorem puts a convergence rate bound on the approximate MD scheme \cref{eq:approximateMD} in terms of the forward-backward inconsistency. More precisely, the inconsistency is quantified by the difference of the iterates in the dual space. This will allow us to show approximate convergence when the inverse mirror map is not known exactly.

\begin{theorem}[Regret Bound for Approximate MD]\label{thm:approxMD}
Suppose $f$ is $\mu$-strongly convex with parameter $\mu>0$, and $\Psi$ is a mirror potential with strong convexity parameter $\sigma$. Let $\{\tilde{x}_k\}_{k=0}^\infty$ be some sequence in $\mathcal{X} = \R^n$, and $\{x_k\}_{k=1}^\infty$ be the corresponding exact MD iterates generated by \cref{eq:trueMDiterates}. We have the following regret bound:
\begin{equation}
    \begin{split}
        &\sum_{k=1}^K t_k(f(\tilde{x}_k) - f(x^*))  \\
        &\le B(x^*, \tilde{x}_1) + \sum_{k=1}^K \left[\frac{1}{\sigma} t_k^2 \|\nabla f(\tilde{x}_k)\|_*^2 + \left(\frac{1}{2t_k\mu} + \frac{1}{\sigma}\right)\|\nabla \Psi(\tilde{x}_{k+1}) - \nabla \Psi(x_{k+1})\|_*^2 \right].
    \end{split}
\end{equation}
\end{theorem}
\begin{proof}
We start by employing \textit{amortization} to find an upper bound on the following expression:
\begin{equation}\label{eq:amortization}
    t_k f(\tilde{x}_k) - t_k f(x^*) + (B(x^*, \tilde{x}_{k+1}) - B(x^*, \tilde{x}_{k})).
\end{equation}%
From the formulation \eqref{eq:approximateMD}, since $\nabla \Psi^* = (\nabla \Psi)^{-1}$:
\[\nabla \Psi(x_{k+1}) = \nabla\Psi(\tilde{x}_k) - t_k \nabla f(\tilde{x}_k).\]

We have the following bound on $B(x^*, \tilde{x}_{k+1}) - B(x^*, \tilde{x}_{k})$:

\begin{align*}
    &\quad B(x^*, \tilde{x}_{k+1}) - B(x^*, \tilde{x}_{k}) = \Psi(x^*) - \Psi(\tilde{x}_{k+1}) - \langle\nabla \Psi(\tilde{x}_{k+1}) ,x^* - \tilde{x}_{k+1}\rangle && \\
    &\qquad - \left[\Psi(x^*) - \Psi(\tilde{x}_{k}) - \langle\nabla \Psi(\tilde{x}_{k}) ,x^* - \tilde{x}_{k}\rangle\right] &&[\text{definition of }B] \\
    &= \Psi(\tilde{x}_{k}) - \Psi(\tilde{x}_{k+1}) - \langle\nabla \Psi(\tilde{x}_{k+1}) ,x^* - \tilde{x}_{k+1}\rangle + \langle\nabla \Psi(\tilde{x}_{k}) ,x^* - \tilde{x}_{k}\rangle &&\text{[cancel $\Psi(x^*)$]} \\
    &= \Psi(\tilde{x}_{k}) - \Psi(\tilde{x}_{k+1}) - \textcolor{blue}{\langle  \nabla\Psi (x_{k+1}), x^* - \tilde{x}_{k+1} \rangle} && \text{[add/subtract}\\
    &\qquad - \langle \nabla \Psi(\tilde{x}_{k+1}) - \textcolor{blue}{\nabla \Psi(x_{k+1})}, x^* - \tilde{x}_{k+1} \rangle + \langle \nabla \Psi(\tilde{x}_k), x^* - \tilde{x}_k\rangle && \text{terms in blue]}\\
    &= \Psi(\tilde{x}_{k}) - \Psi(\tilde{x}_{k+1}) - \langle \nabla \Psi(\tilde{x}_k) - t_k \nabla f(\tilde{x}_k), x^* - \tilde{x}_{k+1} \rangle && \text{[MD update \eqref{eq:trueMDiterates}}\\
    &\qquad - \langle \nabla \Psi(\tilde{x}_{k+1})  - \nabla \Psi(x_{k+1}), x^* - \tilde{x}_{k+1} \rangle + \langle \nabla \Psi(\tilde{x}_k), x^* - \tilde{x}_k\rangle && \text{ on } \nabla \Psi(x_{k+1})]\\
    &= \underbrace{\Psi(\tilde{x}_k) - \Psi(\tilde{x}_{k+1}) + \langle \nabla \Psi(\tilde{x}_k), \tilde{x}_{k+1} - \tilde{x}_k \rangle}_{-B_{\Psi}(\tilde{x}_{k+1}, \tilde{x}_k)} \\
    &\qquad + \langle t_k \nabla f(\tilde{x}_k),  x^* - \tilde{x}_{k+1} \rangle  - \langle \nabla \Psi(\tilde{x}_{k+1}) - \nabla \Psi(x_{k+1}), x^* - \tilde{x}_{k+1} \rangle.
\end{align*}
Observe that the first line in the final expression is precisely $-B_\Psi(\tilde{x}_{k+1}, \tilde{x}_{k})$. By $\sigma$-strong-convexity of $\Psi$, we have $-B_\Psi(\tilde{x}_{k+1}, \tilde{x}_{k}) \le -\frac{\sigma}{2} \|\tilde{x}_{k+1} - \tilde{x}_{k}\|^2$. Therefore, our final bound for this expression is:
\begin{equation}\label{eq:BregmanBound}
    \begin{split}
        &\quad B(x^*, \tilde{x}_{k+1}) - B(x^*, \tilde{x}_{k}) \\
        &\le -\frac{\sigma}{2} \|\tilde{x}_{k+1} - \tilde{x}_{k}\|^2 + {\langle t_k \nabla f(\tilde{x}_k),  x^* - \tilde{x}_{k+1} \rangle} - \langle \nabla \Psi(\tilde{x}_{k+1}) - \nabla \Psi(x_{k+1}), x^* - \tilde{x}_{k+1} \rangle.
    \end{split}
\end{equation}%
Returning to bounding the initial expression \eqref{eq:amortization}, we have by substituting \eqref{eq:BregmanBound}:

\begin{align*}
    &\quad t_k f(\tilde{x}_k) - t_k f(x^*) + (B(x^*, \tilde{x}_{k+1}) - B(x^*, \tilde{x}_{k})) \\
    &\le t_k f(\tilde{x}_k) - t_k f(x^*) + {\langle t_k \nabla f(\tilde{x}_k), x^* - \tilde{x}_{k+1}\rangle} \\
    & \qquad - \frac{\sigma}{2}\|\tilde{x}_{k+1} - \tilde{x}_{k}\|^2 - \langle \nabla \Psi(\tilde{x}_{k+1}) - \nabla \Psi(x_{k+1}), x^* - \tilde{x}_{k+1} \rangle && \text{[by \eqref{eq:BregmanBound}]} \\
    &= t_k f(\tilde{x}_k) - t_k f(x^*) + {\langle t_k \nabla f(\tilde{x}_k), x^* - \textcolor{blue}{\tilde{x}_k\rangle}}  + {\langle t_k \nabla f(\tilde{x}_k) , \textcolor{blue}{\tilde{x}_k} - \tilde {x}_{k+1} \rangle} && [\text{add/subtract}\\
    & \qquad - \frac{\sigma}{2}\|\tilde{x}_{k+1} - \tilde{x}_{k}\|^2 - \langle \nabla \Psi(\tilde{x}_{k+1}) - \nabla \Psi(x_{k+1}), x^* - \tilde{x}_{k+1} \rangle && \text{terms in blue}]\\
    &= -t_k B_f(x^*, \tilde{x}_k)  + \langle t_k \nabla f(\tilde{x}_k) , \tilde{x}_k - \tilde{x}_{k+1} \rangle\\
    & \quad - \frac{\sigma}{2}\|\tilde{x}_{k+1} - \tilde{x}_{k}\|^2 - \langle \nabla \Psi(\tilde{x}_{k+1}) - \nabla \Psi(x_{k+1}), x^* - \textcolor{blue}{\tilde{x}_{k}} \rangle \\
    & \quad - \langle \nabla \Psi(\tilde{x}_{k+1}) - \nabla \Psi(x_{k+1}), \textcolor{blue}{\tilde{x}_k} - \tilde{x}_{k+1} \rangle. 
\end{align*}
The above two equalities are obtained by writing the second term of the inner products as $x^* - \tilde{x}_{k+1} = (x^* - \tilde{x}_{k}) + (\tilde{x}_{k} - \tilde{x}_{k+1})$, and by the definition of $B_f$. By $\mu$-strong-convexity of $f$, we get $-t_k B_f(x^*, \tilde{x}_k) \le -\frac{t_k \mu}{2} \|x^* - \tilde{x}_k\|^2$. Therefore, the bound on the quantity in \eqref{eq:amortization} reduces to
\begin{equation}
\begin{split}
    &\quad t_k f(\tilde{x}_k) - t_k f(x^*) + (B(x^*, \tilde{x}_{k+1}) - B(x^*, \tilde{x}_{k})) \\
    &\le - \frac{t_k\mu}{2} \|x^*- \tilde{x}_k\|^2  + \langle t_k \nabla f(\tilde{x}_k) , \tilde{x}_k - \tilde{x}_{k+1} \rangle \\
    & \quad - \frac{\sigma}{2}\|\tilde{x}_{k+1} - \tilde{x}_{k}\|^2 - \langle \nabla \Psi(\tilde{x}_{k+1}) - \nabla \Psi(x_{k+1}), x^* - \tilde{x}_{k} \rangle  \\ & \quad -\langle \nabla \Psi(\tilde{x}_{k+1}) - \nabla \Psi(x_{k+1}), \tilde{x}_k - \tilde{x}_{k+1} \rangle.
\end{split}
\end{equation}
Liberally applying Cauchy-Schwarz and Young's inequality to bound the inner product terms:
\begin{equation}
\begin{split}
    &\quad t_k f(\tilde{x}_k) - t_k f(x^*) + (B(x^*, \tilde{x}_{k+1}) - B(x^*, \tilde{x}_{k})) \\
    &\le - \frac{t_k\mu}{2} \|x^*- \tilde{x}_k\|^2 + \frac{1}{\sigma} t_k^2 \|\nabla f(\tilde{x}_k)\|_*^2 + \frac{\sigma}{4} \|\tilde{x}_k - \tilde{x}_{k+1}\|^2 \\
    & \quad - \frac{\sigma}{2}\|\tilde{x}_{k+1} - \tilde{x}_{k}\|^2 
    + \frac{1}{2 t_k \mu}\|\nabla \Psi(\tilde{x}_{k+1}) - \nabla \Psi(x_{k+1})\|_*^2 + \frac{t_k\mu}{2}\|x^* - \tilde{x}_{k}\|^2 \\
    & \quad + \frac{1}{\sigma}\|\nabla \Psi(\tilde{x}_{k+1}) - \nabla \Psi(x_{k+1})\|_*^2 + \frac{\sigma}{4}\| \tilde{x}_k - \tilde{x}_{k+1}\|^2 \\
    &\le \frac{1}{\sigma} t_k^2 \|\nabla f(\tilde{x}_k)\|_*^2 + \left(\frac{1}{2t_k\mu} + \frac{1}{\sigma}\right)\|\nabla \Psi(\tilde{x}_{k+1}) - \nabla \Psi(x_{k+1})\|_*^2.
\end{split}
\end{equation}
Summing from $k=1$ to $K$, we get
\begin{equation}\label{eq:ApproxMDRegret}
    \begin{split}
        &\quad \sum_{k=1}^K \left[t_k f(\tilde{x}_k) - t_k f(x^*) + (B(x^*, \tilde{x}_{k+1}) - B(x^*, \tilde{x}_{k}))\right] \\
        & \le \sum_{k=1}^K \left[ \frac{1}{\sigma} t_k^2 \|\nabla f(\tilde{x}_k)\|_*^2 + \left(\frac{1}{2t_k\mu} + \frac{1}{\sigma}\right)\|\nabla \Psi(\tilde{x}_{k+1}) - \nabla \Psi(x_{k+1})\|_*^2\right].
    \end{split}
\end{equation}
Observe $\sum_{k=1}^K (B(x^*, \tilde{x}_{k+1}) - B(x^*, \tilde{x}_{k})) = B(x^*, \tilde{x}_{K+1}) - B(x^*, \tilde{x}_{1}) \ge - B(x^*, \tilde{x}_{1})$. Apply this with \cref{eq:ApproxMDRegret} to finish the regret bound.
\end{proof}
\begin{remark}
This bound may be extended to the constrained case $\mathcal{X} \subsetneq \R^n$. This can be shown by adding an extra projection step to the iterates of the form $\pi(y) = \argmin_{x \in \mathcal{X}} B(x,y)$, and having $\tilde{x}_{k+1}$ instead approximate the projection of the exact mirror step $\tilde{x}_{k+1} \approx \pi(x_{k+1})$ in \eqref{eq:approximateMD} \cite{nemirovsky1983MD}. Note that if $y \notin \mathcal{X}$, then $B(x^*,\pi(y)) \le B(x^*, y)$ for any $x^* \in \mathcal{X}$.
\end{remark}
\begin{remark}
The convex function $f$ need not be differentiable, and having a non-empty subgradient at every point is sufficient for the regret bound to hold. The proof will still work if $\nabla f$ is replaced by a subgradient $f' \in \partial f$.
\end{remark}
\begin{remark}\label{rem:smallSS}
Observe there is a $t_k^{-1}$ coefficient in the approximation term. This prevents us from taking $t_k \searrow 0$ to get convergence as in the classical MD case. Intuitively, a sufficiently large gradient step is required to correct for the approximation. However, due to the Lipschitz condition on the objective $f$, the gradient step is still required to be limited above for convergence. \sm{Maybe mention that there is a trade-off here.} \hy{added}
\end{remark}

With \cref{thm:approxMD}, we no longer require precise knowledge of the convex conjugate. In particular, this allows us to parameterize the forward mirror potential with an \ac{ICNN}, for which there is no closed-form convex conjugate in general. We are thus able to approximate the backwards mirror potential with another neural network, while maintaining approximate convergence guarantees. While the true backward potential will be convex, these results allow us to use a non-convex network, resulting in better numerical performance. 
\sm{You might want to write it differently. The optimal backward map will always have a convex potential. The advantage of using a generic neural net modeling is rather empirical. Say something like `this was found to be better numerically'.} \hy{reworded}

\subsection{Relative Smoothness Assumption}
We have seen that we can approximate the iterations of \ac{MD} and still obtain convergence guarantees. With the slightly weaker assumption of relative smoothness and relative strong convexity, \ac{MD} can be shown to converge \cite{lu2018relreg}. We can get a similar and cleaner bound by slightly modifying the proof of convergence for classical \ac{MD} under these new assumptions. 

\begin{definition}[Relative Smoothness/Convexity]
Let $\Psi:\mathcal{X} \rightarrow \R$ be a differentiable convex function, defined on a convex set $\mathcal{X}$ (with non-empty interior), which will be used as a reference. Let $f:\mathcal{X} \rightarrow \R$ be another differentiable convex function.

$f$ is \emph{$L$-smooth relative to $\Psi$} if for any $x,y \in \text{int}(\mathcal{X})$,
\begin{equation}
    f(y) \le f(x) + \langle \nabla f(x), y-x \rangle + L B_\Psi(y,x).
\end{equation}

$f$ is \emph{$\mu$-strongly-convex relative to $\Psi$} if for any $x,y \in \text{int}(\mathcal{X})$,
\begin{equation}
    f(y) \ge f(x) + \langle \nabla f(x), y-x \rangle + \mu B_\Psi(y,x).
\end{equation}
\end{definition}

Observe that these definitions of relative smoothness and relative strong convexity extend the usual notions of $L$-smoothness and strong convexity with the Euclidean norm by taking $\Psi= \frac{1}{2}\|\cdot\|_2^2$, recovering $B_\Psi(x,y) = \frac{1}{2}\|x-y\|_2^2$. Moreover, if $\nabla f$ is $L$-Lipschitz and $\Psi$ is $\mu$-strongly convex with $\mu>0$, then $f$ is $L/\mu$ smooth relative to $\Psi$. If both functions are twice-differentiable, the above definitions are equivalent to the following \cite[Prop 1.1]{lu2018relreg}:
\begin{equation}
    \mu \nabla^2 \Psi \preceq \nabla^2 f \preceq L\nabla^2 \Psi.
\end{equation}

Using the relative smoothness and relative strong convexity conditions, we can show convergence even when the convex objective function $f$ is flat, as long as our mirror potential $\Psi$ is also flat at those points. The analysis given in \cite{lu2018relreg} readily extends to the case where our iterations are approximate.

\begin{theorem}\label{thm:relativeMD}
Let $f$ be relatively $L$-smooth and relatively $\mu$-strongly-convex with respect to the mirror map $\Psi$, with $L>0,\, \mu \ge 0$. Let $\{\tilde{x}_k\}_{k\ge 0}$ be a sequence in $\mathcal{X}$, and consider the iterations $\{x_k\}_{k\ge 1}$ defined as%
\sm{The theorem extends convergence results of MD for relatively smooth functions, but how do we know that $f$ is relatively smooth wrt $\Psi$? Given an ICNN $\Psi$, this isn't verifiable, right?}\hy{in practice, relative smoothness can be derived from lipschitz constant of f and strong convexity of mu: if $f$ is $L_f$ lipschitz and $\Psi$ is $\mu$-strongly convex then $f$ is $L_f/\mu$ smooth wrt $\Psi$}
\begin{equation}
        x_{k+1} = \arg \min_{x \in \mathcal{X}}\left\{\langle x, \nabla f(\tilde{x}_k) \rangle + L B(x, \tilde{x}_k)\right\},
\end{equation}
i.e. the result of applying a single MD update step with fixed step size $1/L$ to each $\tilde{x}_k$. We have the following bound (for any $x \in \mathcal{X}$), where the middle expression is discarded if $\mu=0$:

\begin{equation}\label{eq:relativeSmoothMDBound}
    \min_{1\le i \le k} f(\tilde{x}_i) - f(x) \le \frac{\mu B(x, \tilde{x}_0)}{(1+\frac{\mu}{L-\mu})^k -1} + M_k \le \frac{L-\mu}{k}B(x, \tilde{x}_0) + M_k,
\end{equation}
where 
\begin{equation}
    M_k = \frac{\sum_{i=1}^k (\frac{L}{L-\mu})^i [L \langle \nabla \Psi(x_i) - \nabla \Psi(\tilde{x}_i), x-\tilde{x}_i \rangle + \langle \nabla f(x_i), \tilde{x}_i - x_i \rangle]}{\sum_{i=1}^k (\frac{L}{L-\mu})^i}.
\end{equation}
In particular, if $L \langle \nabla \Psi(x_i) - \nabla \Psi(\tilde{x}_i), x-\tilde{x}_i \rangle + \langle \nabla f(x_i), \tilde{x}_i - x_i \rangle$ is uniformly bounded (from above) by $M$, we can replace $M_k$ by $M$ in \eqref{eq:relativeSmoothMDBound}. 
\end{theorem}
\begin{proof}
We follow the proof of \cite[Thm 3.1]{lu2018relreg} very closely. We state first the three-point property (\cite[Lemma 3.1]{lu2018relreg}, \cite{tseng08}). 
\begin{lemma}[Three-point property]
Let $\phi(x)$ be a proper l.s.c. convex function. If 
\[z_+ = \argmin_{x} \{\phi(x) + B(x, z)\},\]
then 
\[\phi(x) + B(x,z) \ge \phi(z_+) + B(z_+, z) + B(x, z_+),\quad \text{for all } x \in \mathcal{X}.\]
\end{lemma}
As in \cite[Eq 28]{lu2018relreg}, we have for any $x \in \mathcal{X}$ and $i\ge 1$,
\begin{equation}
    \begin{split}
        f(x_i) &\le f(\tilde{x}_{i-1}) + \langle \nabla f(\tilde{x}_{i-1}), x_i - \tilde{x}_{i-1})\rangle + L B(x_i, \tilde{x}_{i-1})\\
        & \le f(\tilde{x}_{i-1}) + \langle \nabla f(\tilde{x}_{i-1}), x - \tilde{x}_{i-1})\rangle + L B(x, \tilde{x}_{i-1}) - L B(x, x_i) \\
        & \le f(x) + (L-\mu) B(x, \tilde{x}_{i-1}) - L B(x, x_i).
    \end{split}
\end{equation}
The first inequality follows from $L$-smoothness relative to $\Psi$, the second inequality from the three-point property applied to $\phi(x) = \frac{1}{L} \langle \nabla f(\tilde{x}_{i-1}), x-\tilde{x}_{i-1}\rangle$ and $z=\tilde{x}_{i-1}, z_+ = x_i$, and the last inequality from $\mu$-strong-convexity of $f$ relative to $\Psi$. We thus have
\begin{equation}
    \begin{split}
        f(\tilde{x}_i) &= f(x_i) + f(\tilde{x}_i) - f(x_i) \\
        &\le f(x) + (L-\mu) B(x, \tilde{x}_{i-1}) - LB(x, x_i) + f(\tilde{x}_i) - f(x_i) \\
        &= (L-\mu) B(x, \tilde{x}_{i-1}) - LB(x, \tilde{x}_i) \\
        &\quad + [f(x) + LB(x, \tilde{x}_i) - LB(x, x_i) + f(\tilde{x}_i) - f(x_i)].
    \end{split}
\end{equation}

By induction/telescoping, we get:
\begin{equation}
\begin{split}
        \sum_{i=1}^k \left(\frac{L}{L-\mu}\right)^i f(\tilde{x}_i) &\le LB(x, \tilde{x}_0) + \sum_{i=1}^k \left(\frac{L}{L-\mu}\right)^i f(x)\\
        &\quad + \sum_{i=1}^k \left(\frac{L}{L-\mu}\right)^i[L(B(x, \tilde{x}_i) -B(x, x_i)) + f(\tilde{x}_i) - f(x_i)].
\end{split}
\end{equation}
The final ``approximation error" term is
\begin{equation}
    \begin{split}
        &L(B(x, \tilde{x}_i) - B(x, x_i)) + f(\tilde{x}_i) - f(x_i)\\
        =&\ L \langle \nabla \Psi(x_i) - \nabla \Psi(\tilde{x}_i), x-\tilde{x}_i \rangle - LB(\tilde{x_i}, x_i) + f(\tilde{x}_i) - f(x_i)\\
        \le&\ L \langle \nabla \Psi(x_i) - \nabla \Psi(\tilde{x}_i), x-\tilde{x}_i \rangle - B_f(\tilde{x_i}, x_i) + f(\tilde{x}_i) - f(x_i) \\
        =&\ L \langle \nabla \Psi(x_i) - \nabla \Psi(\tilde{x}_i), x-\tilde{x}_i \rangle + \langle \nabla f(x_i), \tilde{x}_i - x_i \rangle,
    \end{split}
\end{equation}
where in the inequality, we use the definition of $L$-relative smoothness $B_f(x,y) \le LB_\Psi(x,y)$.
(Recall $B(c,a) + B(a,b) - B(c,b) = \langle \nabla \Psi(b) - \nabla \Psi(a), c-a \rangle$ \cite[Lemma 4.1]{BECK2003167}.)

Substituting $C_k$ defined by 
\[\sum_{i=1}^k \left(\frac{L}{L-\mu}\right)^i \eqqcolon \frac{1}{C_k}\]
and rearranging, we get
\begin{equation}
\begin{aligned}
    &\min_{1\le i\le k} f(\tilde{x}_i) - f(x) \le  C_k L B(x, \tilde{x}_0) \\
    & \qquad + C_k \sum_{i=1}^k \left(\frac{L}{L-\mu}\right)^i [L \langle \nabla \Psi(x_i) - \nabla \Psi(\tilde{x}_i), x-\tilde{x}_i \rangle + \langle \nabla f(x_i), \tilde{x}_i - x_i \rangle].
\end{aligned}
\end{equation}

In particular, if we have a uniform bound on $[L \langle \nabla \Psi(x_i) - \nabla \Psi(\tilde{x}_i), x-\tilde{x}_i \rangle + \langle \nabla f(x_i), \tilde{x}_i - x_i \rangle]$, say $M$, then we have 
\begin{equation}
    \min_{1\le i\le k} f(\tilde{x}_i) - f(x) \le C_k L B(x, \tilde{x}_0) + M.
\end{equation}

Finally, note that if $\mu=0$ then $C_k = 1/k$, and if $\mu > 0$ then
\[C_k = \left(\sum_{i=1}^k \left(\frac{L}{L-\mu}\right)^i\right)^{-1} = \frac{\mu}{L\left((1+\frac{\mu}{L-\mu})^k -1\right)}\le 1/k.\]

\end{proof}

\Cref{thm:relativeMD} gives us convergence rate bounds up to an additive approximation error $M_k$, depending on how far the approximate iterates $\tilde{x}_k$ are from the true MD iterates $x_k$. By taking $x$ in \cref{eq:relativeSmoothMDBound} to be an optimal point $x^*$ where $f$ attains its minimum, we can get approximate linear convergence and approximate $O(1/k)$ convergence if the relative strong convexity parameters satisfy $\mu>0$ and $\mu=0$ respectively. In particular, the quantity 
\begin{equation}\label{eq:relativeQuantity}
    L \langle \nabla \Psi(x_i) - \nabla \Psi(\tilde{x}_i), x-\tilde{x}_i \rangle + \langle \nabla f(x_i), \tilde{x}_i - x_i \rangle
\end{equation}
that we would like to bound gives an interpretation in terms of how the approximate iterates $\tilde{x}_i$ should be close to $x_i$. To minimize the first term, $\nabla \Psi(x_i) - \nabla \Psi(\tilde{x}_i)$ should be small, and $\tilde{x}_i - x_i$ should be small to minimize the second term. 

\subsection{Training Procedure}
In this section, we will outline our general training procedure and further detail our definitions for having faster convergence. We further propose a loss function to train the mirror potentials $M_\theta$ and $M_\vartheta^*$ to enforce both faster convergence, as well as forward-backward consistency in order to apply \Cref{thm:approxMD}.

Suppose we have a fixed function class $\mathcal{F}$ consisting of convex functions $f:\mathcal{X} \rightarrow \R$, where $\mathcal{X}\subseteq \R^d$ is some convex set that we wish to optimize over. Our goal is to efficiently minimize typical functions in $\mathcal{F}$ by using our learned mirror descent scheme. 

For a function $f \in \mathcal{F}$, suppose we have data initializations $x \in \mathcal{X}$ drawn from a data distribution $\mathbb{P}_{x|f}$, possibly depending on our function. Let $\{\tilde{x}_k\}_{k=1}^K$ be the sequence constructed by applying learned mirror descent with forward potential $M_\theta$ and backward potential $M_\vartheta^*$, with initialization $\tilde{x}_0=x$:
\begin{equation}
    \tilde{x}_{k+1} = \nabla M_\vartheta^* (\nabla M_\theta (\tilde{x}_k) - t_k \nabla f(\tilde{x}_k)).
\end{equation}%


To parameterize our mirror potentials $M_\theta, M_\vartheta^*:\R^d\rightarrow \R$, we use the architecture proposed by Amos et al. for an input convex neural network (ICNN) \cite{amos2017icnn}. The input convex neural networks are of the following form:
\begin{equation}
    z_{i+1} = \sigma \left(W_i^{(z)} z_i + W_i^{(x)}x + b_i\right),\quad M(x;\theta) = z_l,
\end{equation}
where $\sigma$ is the leaky-ReLU activation function, and $\theta = \{W_{0:l-1}^{(x)}, W_{1:l-1}^{(z)}, b_{0:l-1}\}$ are the parameters of the network. For the forward mirror potential $M_\theta$, we clip the weights such that $W_i^{(z)}$ are non-negative, so the network is convex in its input $x$ \cite[Prop 1]{amos2017icnn}. This can be done for both fully connected and convolutional layers. We note that it is not necessary for the backwards mirror potential $M_\vartheta^*$ to be convex, which allows for more expressivity. Using the ICNN architecture allows for guaranteed convex mirror potentials with minimal computational overhead. By adding an additional small quadratic term $\mu\|x\|^2$ to the ICNN, we are able to enforce strong convexity of the mirror map as well. 

We would like to enforce that $f(\tilde{x}_k)$ is minimized quickly \emph{on average}, over both the function class and the distribution of initializations $\tilde{x}_0=x$ corresponding to each individual $f$. One possible method is to consider the value of the loss function at or up to a particular iteration $\tilde{x}_N$ for fixed $N$. We also apply a soft penalty such that $\nabla M_\vartheta^* \approx (\nabla M_\theta)^{-1}$ in order to maintain reasonable convergence guarantees. The loss that we would hence like to optimize over the neural network parameter space $(\theta, \vartheta) \in \Theta$ is thus:
\begin{equation}\label{eq:MDTrainLoss}
    \argmin_{\theta,\vartheta}  \mathbb{E}_{f,x}[f(\tilde{x}_N)] + \mathbb{E}_{\mathcal{X}}[\|\nabla M_\vartheta^* \circ \nabla M_\theta - I\|].
\end{equation}
The expectations on the first term are taken over the function class, and further on the initialization distribution conditioned on our function instance. To empirically speed up training, we find it effective to track the loss at each stage, similar to Andrychowicz et al. \cite{andrychowicz2016GDbyGD}. Moreover, it is impractical to have a consistency loss for the entire space $\mathcal{X}$, so we instead limit it to around the samples that are attained. The loss functions that we use will be variants of the following:
\begin{subequations}
\begin{align}
        \tilde{x}_{k+1} &= \nabla M_\vartheta^* (\nabla M_\theta (\tilde{x}_k) - t_k \nabla f(\tilde{x}_k)) \label{},\\
        L(\theta, \vartheta) &=  \mathbb{E}_{f,x}\left[\sum_{k=1}^N r_k f(\tilde{x}_k) + s_k\|(\nabla M_\vartheta^* \circ \nabla M_\theta - I)(\tilde{x}_k)\|\right], \label{eq:trainingLoss}
\end{align}
\end{subequations}

\noindent where $r_k,\ s_k\ge 0$ are some arbitrary weights. For training purposes, we took $r_k = r = 1$ as constant throughout, and varied $s_k = s_{\text{epoch}}$ to increase as training progresses. In particular, we will take $s_0=1$, and increase the value every 50 epochs by a factor of $1.05$. To train our mirror maps, we use a Monte Carlo average of \eqref{eq:trainingLoss} over realizations of $f$ and initializations $x_0$ derived from the training data. This empirical average is optimized using the Adam optimizer for the network parameters $\theta, \vartheta$. This can be written as follows for a minibatch $\{f^{(i)}, x_0^{(i)}\}_{i=1}^B$ of size $B$:

\begin{equation}
    \tilde{L}(\theta, \vartheta) = \frac{1}{B}\sum_{i=1}^B\left[\sum_{k=1}^N r_k f^{(i)}(\tilde{x}_k^{(i)}) + s_k\|(\nabla M_\vartheta^* \circ \nabla M_\theta - I)(\tilde{x}_k^{(i)})\|\right].
\end{equation}


The maximum training iteration was taken to be $N=10$, which provided better generalization to further iterations than for smaller $N$. While $N$ could be taken to be larger, this comes at higher computational cost due to the number of MD iterates that need to be computed. We found that endowing $\mathcal{X}=\R^d$ with the $L^1$ norm was more effective than using the Euclidean $L^2$ norm. The aforementioned convergence results can be then applied with respect to the dual norm $\|\cdot\|_* = \|\cdot\|_\infty$. 

We additionally find it useful to allow the step-sizes to vary over each iteration, rather than being fixed. We will refer to the procedure where we additionally learn the step-sizes as \emph{adaptive LMD}. The learned step-sizes have to be clipped to a fixed interval to maintain convergence and prevent instability. The LMD mirror maps are trained under this ``adaptive" setting, and we will have a choice between using the learned step-sizes and using fixed step-sizes when applying LMD on test data. For testing, we will plot the methods applied with multiple step-sizes. These step-sizes are chosen relative to a `base step-size', which is then multiplied by a `step-size multiplier', denoted as `step-size multi' in subsequent figures.

Training of LMD amounts to training the mirror potentials and applying the approximate mirror descent algorithm. For training, target functions are sampled from a training set, for which the loss \cref{eq:trainingLoss} is minimized over the mirror potential parameters $\theta$ and $\vartheta$. After training, testing can be done by applying the approximate mirror descent algorithm \cref{eq:approximateMD} directly with the learned mirror maps, requiring only forward passes through the networks. This allows for efficient forward passes with fixed memory cost, as extra iterates and back-propagation are not required.

All implementations were done in PyTorch, and training was done on Quadro RTX 6000 GPUs with 24GB of memory \cite{PyTorch}. The code for our experiments are publicly available\footnote{\url{https://github.com/hyt35/icnn-md}}.

\section{Learned Mirror Maps With Closed-Form Inverses}
\label{sec:exactMDMaps}
\sm{Use a different title, something like `illustrative examples with simple mirror maps with closed-form backward map'}\hy{changed}
We illustrate the potential use of \ac{LMD} by learning simple mirror maps with closed-form backward maps, and how this can lead to faster convergence rates on certain problems. We demonstrate these maps on two convex problems: solving unconstrained least squares, and training an \ac{SVM} on 50 features. We first mention two functional mirror maps that can be parameterized using neural networks, and describe the training setup in this scenario. 

One possible parameterization of the mirror potential is using a quadratic form. This can be interpreted as gradient descent, with a multiplier in front of the gradient step. The mirror potentials and mirror maps are given as follows, where $x \in \R^d$ and $A \in \R^{d\times d}$:%
\begin{equation}
    \Psi(x) = \frac{1}{2} x^\top A x,\ \nabla \Psi(x) = \left(\frac{1}{2}A+\frac{1}{2}A^\top\right)x,\
    \nabla \Psi^*(y) = \left[\frac{1}{2}A+\frac{1}{2}A^\top\right]^{-1}y.
\end{equation}

The weight matrix $A$ was initialized as $A = I + E$, where $I$ is the identity matrix and $E$ is a diagonal matrix with random $N(0,0.001)$ entries. For $\Psi$ to be strictly convex, the symmetrization $(A+A^\top)/2$ needs to be positive definite. With this initialization of $A$, we numerically found in our example that explicitly enforcing this non-negativity constraint was not necessary, as the weight matrices $A$ automatically satisfied this condition after training. 



Another simple parameterization of the mirror potential is in the form of a neural network with one hidden layer. In particular, we will consider the case where our activation function is a smooth approximation to leaky-ReLU, given by $g(t) \coloneqq \alpha t + (1-\alpha) \log(1+\exp(t))$. Here, the binary operator $\odot$ for two similarly shaped matrices/vectors is the Hadamard product, defined by component-wise multiplication $(x\odot y)_i = x_i y_i$. Operations such as reciprocals, logarithms, exponentials and division applied to vectors are to be taken component-wise. For $x \in \R^d, A \in \R^{d \times d}, w \in \R_+^d$, the maps are given as follows:
\begin{subequations}
\begin{align}
    \Psi(x) &= w^\top g (Ax) = w^\top(\alpha Ax + (1-\alpha) \log (1+\exp(Ax))),\\
    \nabla \Psi(x) &= \alpha A^\top w + (1-\alpha) w \odot \frac{\exp(Ax)}{1+\exp(Ax)},\\
    \nabla \Psi^*(y) &= A^{-1} \log\left(\frac{(1-\alpha)^{-1} w^{-1} \odot (y-\alpha A^\top w)}{1-(1-\alpha)^{-1} w^{-1} \odot (y-\alpha A^\top w)}\right).
\end{align}
\end{subequations}

This is quite a restrictive model for mirror descent, as it requires the perturbed dual vector $(1-\alpha)^{-1} w^{-1} \odot (y-\alpha A^\top w) - \eta \nabla{f}$  to lie component-wise in $(0,1)$ in order for the backward mirror map to make sense. Nevertheless, this can be achieved by clipping the resulting gradient value to an appropriate interval inside $(0,1)$. 

The negative slope parameter was taken to be $\alpha=0.2$. The weight matrix $A$ was initialized as the identity matrix with entry-wise additive Gaussian noise $N(0,0.01)$, and the vector $w$ was initialized entry-wise using a uniform distribution $\text{Unif}(0,1/d)$. 
\subsection{Least Squares}
The first problem class we wish to consider is that of least squares in two dimensions. This was done with the following fixed weight matrix and randomized bias vectors:
\begin{equation}
    \min_{x\in\R^2} \|Wx-b\|_2^2,\ W = \begin{pmatrix}
2 & 1 \\
1 & 2 \\
\end{pmatrix},\ b \in \R^2.
\end{equation}

For training LMD for least squares, the initialization vectors $x$ and target bias vectors $b$ were independently randomly sampled as Gaussian vectors $b,x\sim N(0, I_2)$. The function class that we wish to optimize over in \eqref{eq:MDTrainLoss} is:
\[\mathcal{F} = \{f_b(x) = \|Wx-b\|_2^2\, : b \in \R^2\},\quad x \sim \mathbb{P}_{x|f} = N(0, I_2),\]
where the expectation $\mathbb{E}_{f,x}$ is taken over $b,x \sim N(0, I_2)$. 

For this problem class, a classical MD algorithm is available. Observe that $\nabla f_b(x) = W^\top W (x - W^{-1}b)$. By taking $\Psi(x) = \frac{1}{2} x^\top (W^\top W) x$, the mirror maps are $\nabla \Psi(x) = (W^\top W) x$, $\nabla \Psi^*(x) = (W^\top W)^{-1} x$. The MD update step \eqref{eq:MDDef} applied to $f = f_b$ becomes
\begin{equation}
    x_{k+1} = (W^\top W)^{-1} ((W^\top W) x_k - t_k \nabla f(x_k)) = x_k - t_k(x_k - W^{-1}b).
\end{equation}
This update step will always point directly towards the true minimizer $W^{-1}b$, attaining linear convergence with appropriate step-size. In \Cref{fig:lsqmdlsqloss,fig:lsqmdnetloss}, this method is added for comparison as the ``MD" method.

We can observe this effect graphically in \Cref{fig:gdmdcurves}. This figure illustrates the effect of MD on changing the optimization path from a curve for GD, to a straight line for MD. Without loss of generality, suppose we take $b=0$ and work in the eigenbasis $\{v_1, v_2\}$ of $W$, so the function to minimize becomes $f(x_1, x_2) = 9x_1^2 + x_2^2$. From initialization $u = (u_1, u_2)$, the gradient flow induces the curve $\gamma(t) = (u_1 \exp(-9t), u_2 \exp(-t))$. The curvature restricts the step-size allowed for gradient descent and moreover increases the curve length compared to the straight MD line, leading to slower convergence. 

An alternative perspective is given using the mirror potential $\Psi$ in \Cref{fig:lsqmdlsqfwd}, which takes the shape of an elliptic paraboloid. In the eigenbasis $v_1 = (1,1),\ v_2 = (1,-1)$ of $W$, the greater curvature of $\Psi$ in the $v_1$ direction implies that gradients are shrunk in this direction in the MD step. In this case, the gradient is shrunk 9 times more in the $v_1$ direction than the $v_2$ direction, inducing the MD curve $\mu(t) = (u_1 \exp(-t), u_2 \exp(-t))$, which is a straight line.

\cref{fig:lsqmdlsq} and \cref{fig:lsqmdnet} illustrate the results of training LMD using the quadratic mirror potential and with the one-layer NN potential respectively. These figures include the evolution of the loss function, the iterates after 10 iterations of adaptive LMD as in the training setting, and a visualization of the mirror map. \cref{fig:lsqmdnetiter} shows the instabilities that occur when the domain of the backwards map is restricted. This is an example of a problem where applying LMD with a well-parameterized mirror map can result in significantly accelerated convergence.


\begin{figure}
\begin{subcolumns}[0.7\textwidth]
  \subfloat[Evolution of least squares loss when using quadratic LMD.]{\includegraphics[width=\subcolumnwidth]{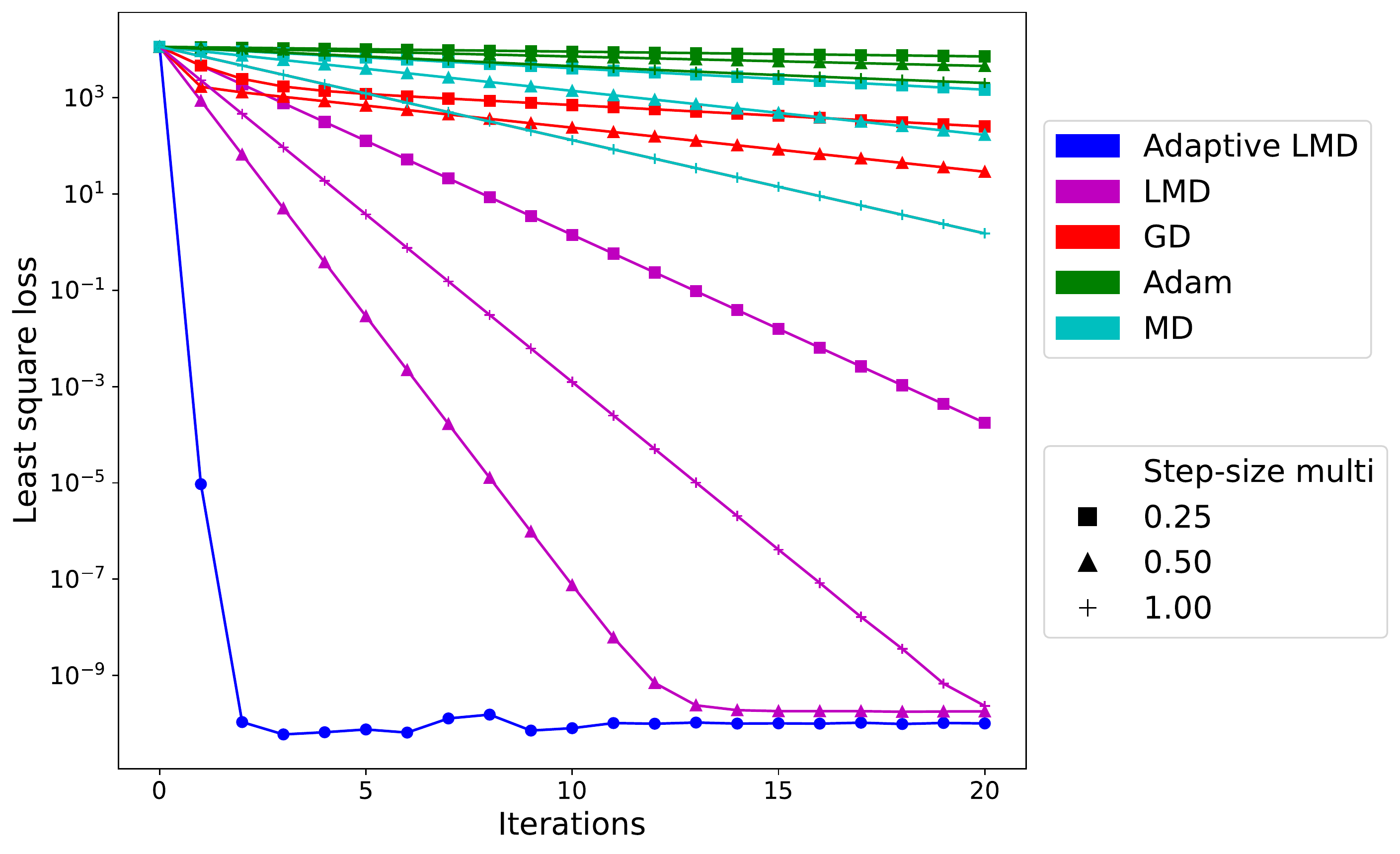}\label{fig:lsqmdlsqloss}}
\nextsubcolumn[0.28\textwidth]
  \subfloat[Optimization path for GD and MD in the eigenbasis of $W$.]{\includegraphics[width=\subcolumnwidth]{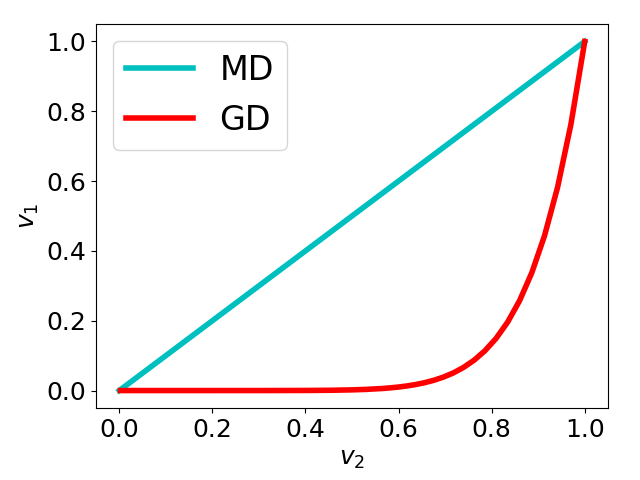}\label{fig:gdmdcurves}}
\nextsubfigure
  \subfloat[Mirror potential $\Psi$]{\includegraphics[width=\subcolumnwidth]{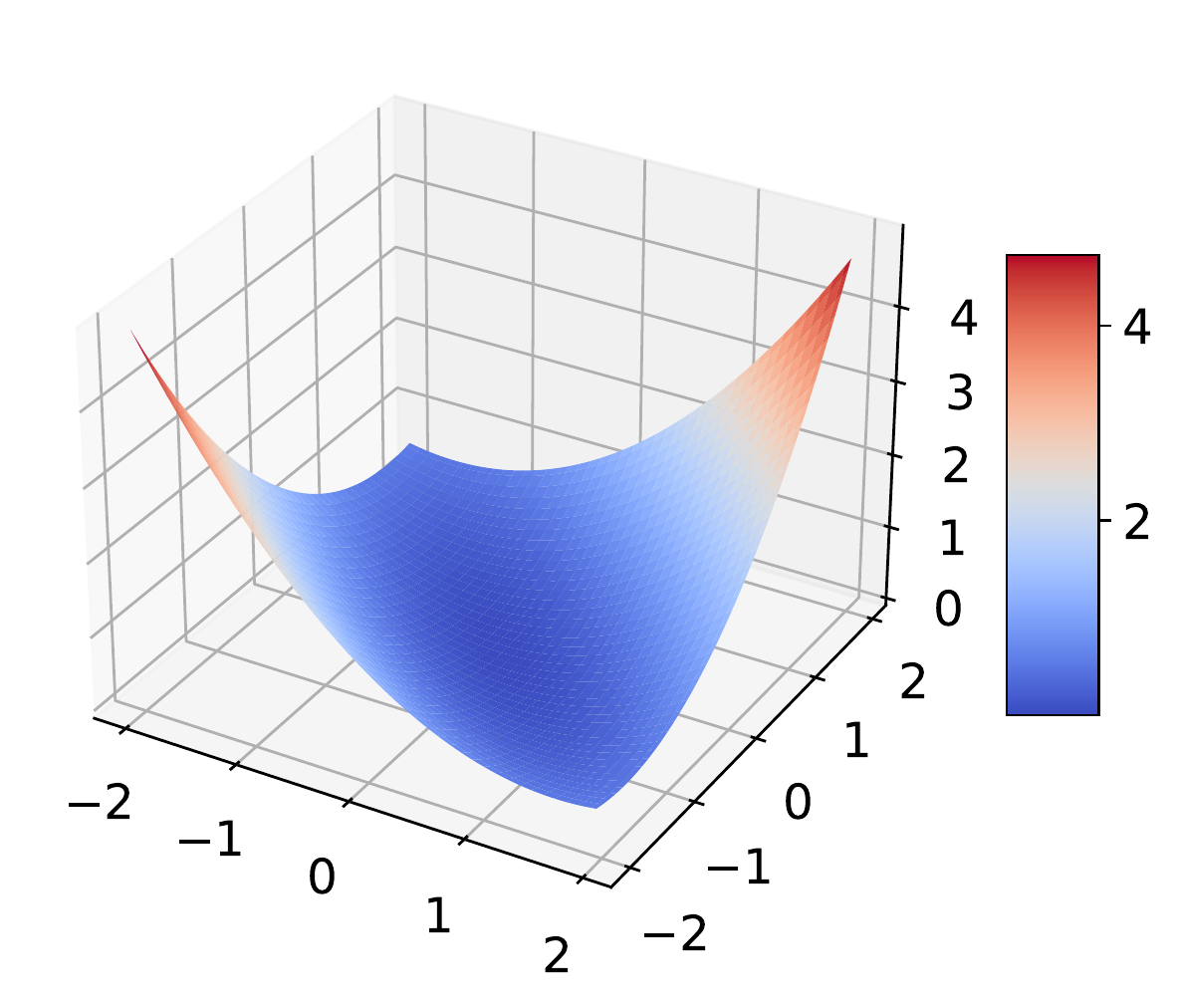}\label{fig:lsqmdlsqfwd}}
\end{subcolumns}
\caption{We observe that quadratic LMD is able to learn a map that allows for linear convergence in this case. Further learning the step-size allows for immediate convergence to machine precision. Note that $W$ has eigenvectors $v_1 = (1,1), v_2 = (1,-1)$ with eigenvalues $\lambda_1 = 3, \lambda_2 = 1$ respectively. As demonstrated in (b), the path that GD takes in the eigenbasis is curved as it minimizes the $v_1$ direction faster than the $v_2$ direction, whereas MD travels in a straight line to the minimizer. This is reflected in (c), where the quadratic form given by $\Psi$ curves more in the $v_1$ direction. Indeed, the learned weight is $A =  \begin{pmatrix}
    0.69 & 0.55 \\
    0.55 & 0.69 
\end{pmatrix}$, which is almost proportional to the classical MD weight $W^\top W = \begin{pmatrix}
    5 & 4 \\
    4 & 5 
\end{pmatrix}$.}
\label{fig:lsqmdlsq}%
\end{figure}

\begin{figure}
\begin{subcolumns}[0.7\textwidth]
  \subfloat[Evolution of least squares loss when using 1-layer-NN LMD.]{\includegraphics[width=\subcolumnwidth]{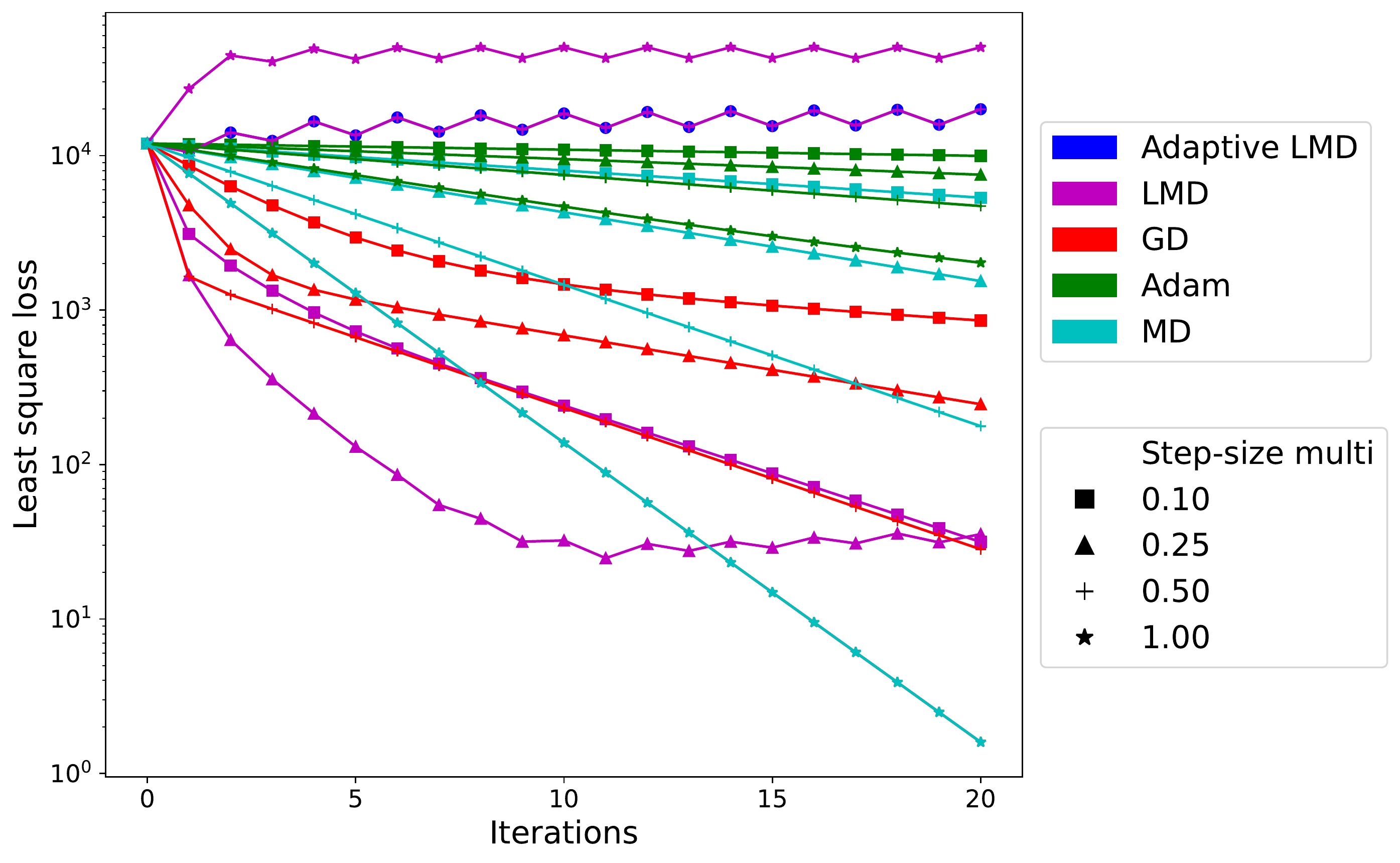}\label{fig:lsqmdnetloss}}
\nextsubcolumn[0.28\textwidth]
  \subfloat[LMD iterations (blue) and true solution (orange)]{\includegraphics[width=\subcolumnwidth]{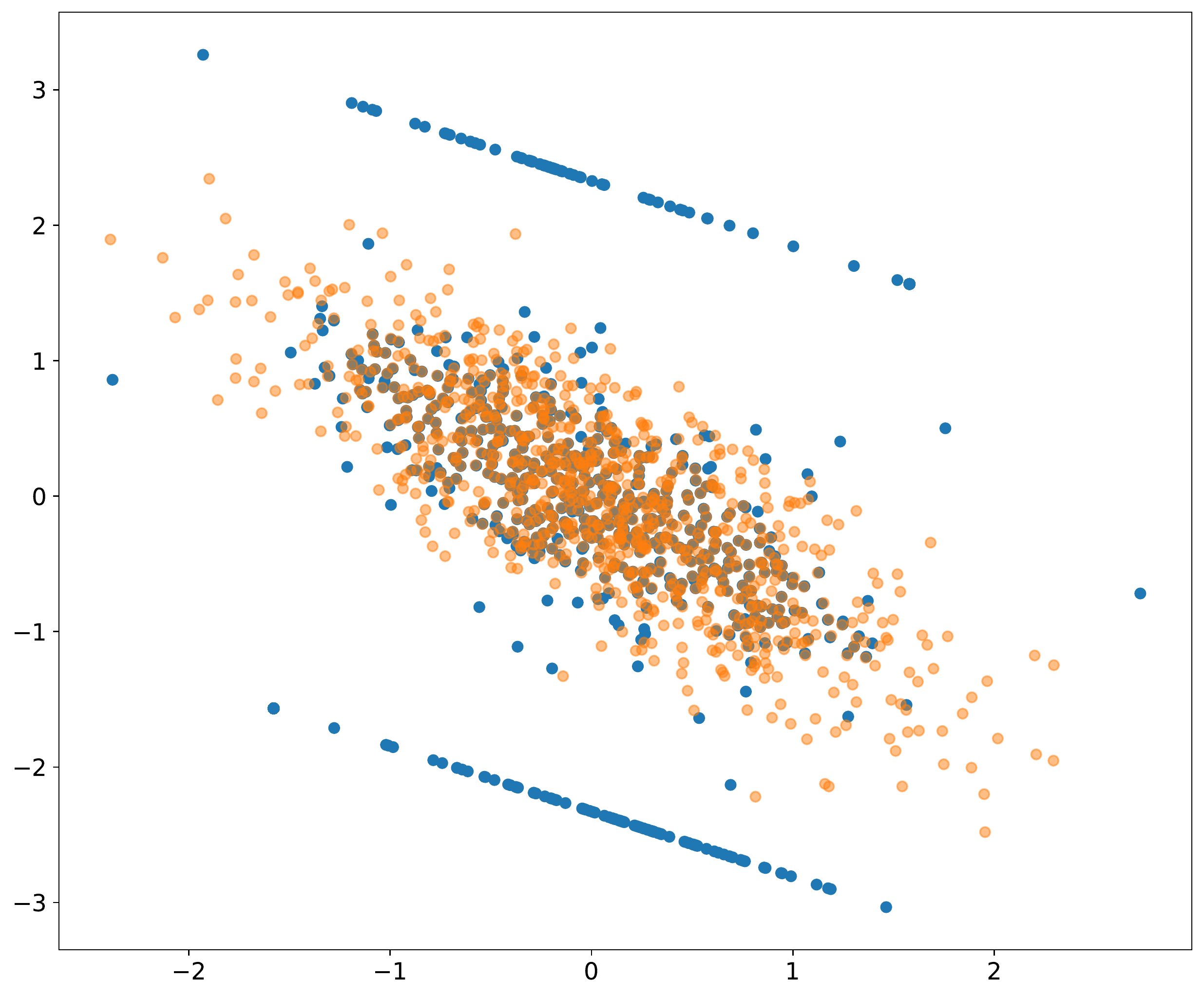}\label{fig:lsqmdnetiter}}
\nextsubfigure
  \subfloat[Mirror potential $\Psi$]{\includegraphics[width=\subcolumnwidth]{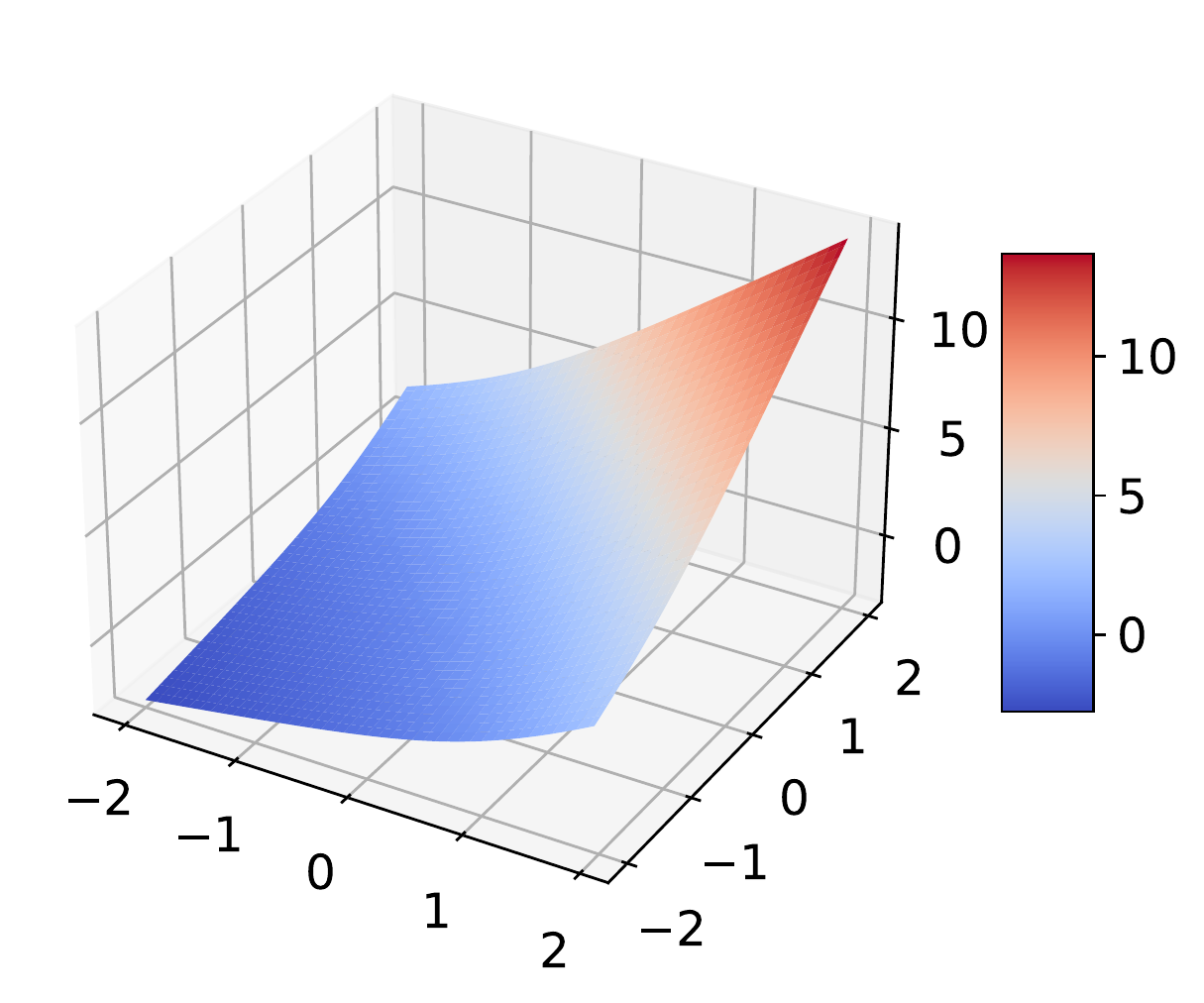}\label{fig:lsqmdnetfwd}}
\end{subcolumns}
\caption{We can see the effect of needing to clip the dual iterates, as it creates a pair of lines (in blue). This heavily affects the performance when using certain step-sizes, and demonstrates the issues with such simple models. Note that the adaptive LMD and LMD with step-size multi $0.5$ are identical. This is due to the choice of interval that the step-size is clipped to be in. The lower bound of the interval coincides with the step-size corresponding to step-size multiplier $0.5$, and adaptive LMD learns the step-sizes to be this lower bound.}%
\label{fig:lsqmdnet}%
\end{figure}


\subsection{SVM}
The second problem class is of training an SVM on the 4 and 9 classes of MNIST. From each image, 50 features were extracted using a small neural network $\phi:[0,1]^{28\times28} \rightarrow \R^{50}$, created by training a neural network to classify MNIST images and removing the final layer. The goal is to train an SVM on these features using the hinge loss; see the SVM formulation in \cref{sec:SVM} for more details. The problem class is of the form 
\[\mathcal{F} = \left\{f_{\mathcal{I}}(\mathbf{w},b) = \frac{1}{2} \mathbf{w}^\top \mathbf{w} + C \sum_{i \in \mathcal{I}} \max(0, 1-y_i (\mathbf{w}^\top \phi_i + b))\right\}.\]
This is the feature class of training SVMs with certain features $\phi_i$ and targets $y_i$, with $i$ taking values in some index set $\mathcal{I}$. In this case, the features and targets were taken as subsets of features extracted from the MNIST dataset. The initializations $(\mathbf{w},b) \sim \mathbb{P}_{(\mathbf{w},b)|f}$ were taken to be element-wise standard Gaussian.

\cref{fig:svmmdlsq} and \cref{fig:svmmdnet} demonstrate the evolution of the SVM hinge loss under the quadratic and one-layer NN mirror potentials respectively. In \cref{fig:svmmdlsq}, the loss evolution under quadratic LMD is faster compared to GD and Adam. This suggests that quadratic LMD can learn features that contribute more to the SVM hinge loss. We can see clearly the effect of learning the step-sizes for increasing convergence rate for the first 10 iterations in the ``adaptive LMD" plot, as well as the effect of a non-optimized step-size after 10 iterations by the increase in loss. In \cref{fig:svmmdnet}, we can see that the one-layer NN mirror potential can perform significantly better than both Adam and GD. However, the instability due to the required clipping causes the hinge loss to increase for larger step-sizes. This instability further motivates the use of using a more expressive neural network, as well as directly modelling the backwards mirror map.



\begin{figure}
\centering
\begin{minipage}[t]{.48\textwidth}
  \centering
  \includegraphics[trim=0 0 170 0, clip, height=5.5cm]{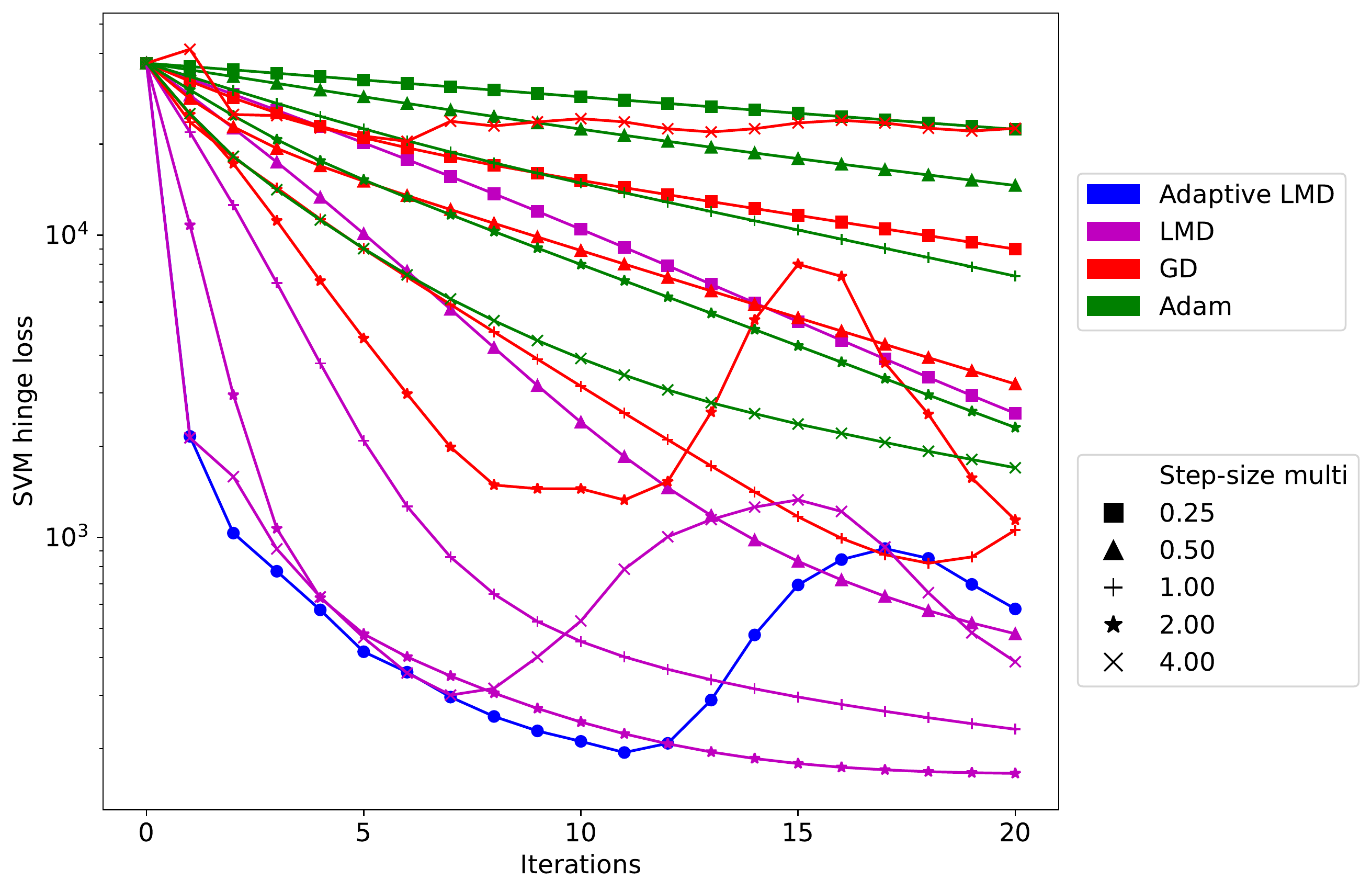}
  \captionof{figure}{Evolution of SVM hinge loss under quadratic LMD. LMD outperforms GD and Adam, with nice convergence for the middle step-size multipliers. With only 3 out of 51 eigenvalues of $A$ being greater than 1 and the rest below $0.5$, this suggests that quadratic LMD is able to learn combinations of features that contribute most to the hinge loss. }
  \label{fig:svmmdlsq}
\end{minipage} \quad \begin{minipage}[t]{.48\textwidth}
  \hspace{-0.5cm}\includegraphics[height=5.5cm,left]{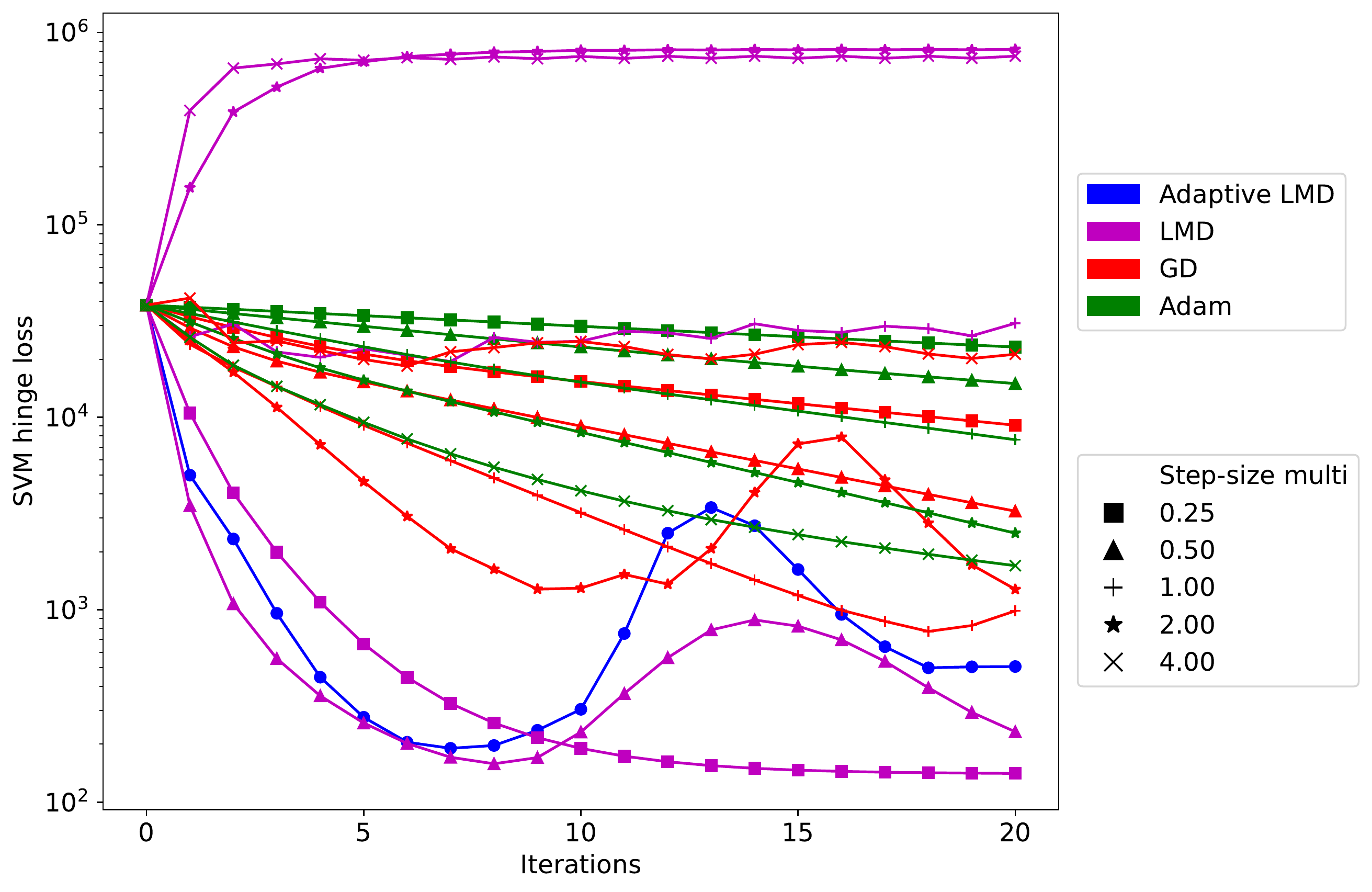}
  \captionof{figure}{1 layer NN mirror map applied to SVM training. In this case, LMD outperforms the other methods for smaller step-sizes. The two LMD lines with higher loss is due to the component-wise clipping that is required for this method.}
  \label{fig:svmmdnet}
\end{minipage}
\end{figure}

Both of these methods require parameterizations of matrices, and moreover require computing the inverse of these matrices, which can cause instability when performing back-propagation. Moreover, such closed form expressions of the convex conjugate are not readily available in general, especially for more complicated mirror potentials parameterized using deep networks. Therefore, training LMD under this setting can not be effectively scaled up to higher dimensions. This motivates our proposed approach and analysis of using two separate networks instead, modeling the mirror and inverse mirror mappings separately.

\section{Numerical Experiments}

\label{sec:experiments}
\sm{Motivated by the examples in the preceding section, we employ the LMD method for a number of convex problems arising in inverse problems and machine learning. Specifically, we use a deep \ac{ICNN} for learning the optimal Bregman distance. However, unlike the constructions in the previous section, the convex conjugate cannot be expressed in closed-form and we resort to approximate inverse of the mirror map modeled by a second neural neural network (which is not necessarily an \ac{ICNN}}\hy{added} 

Motivated by the examples in the preceding section, we employ the LMD method for a number of convex problems arising in inverse problems and machine learning. Specifically, we use a deep \ac{ICNN} for learning the optimal forward mirror potential. However, unlike the constructions in the previous section, the convex conjugate cannot be expressed in a closed form. We instead approximate the inverse of the mirror map using a second neural network, which is not necessarily the gradient of an \ac{ICNN}. We will demonstrate how this can allow for learning the geometry of the underlying problems and result in faster convergence. We will namely be applying the LMD method to the problems of learning a two-class SVM classifier, learning a linear classifier, and model-based denoising and inpainting on STL-10. The dimensionality of these problems, with STL-10 containing images of size $3 \times 96 \times 96$, makes the matrix-based \ac{MD} parameterizations proposed in the previous section infeasible. A list of training and testing hyper-parameters can be found in \cref{tab:hyperparams}. 

\sm{Rewrite this sentence. The method is applied on problems and not on images.}

\sm{Again, for someone who is unfamiliar with your experiments, it is hard to follow what your training objective is. This is a central component of the paper, so take a paragraph to explain all the details.} \hy{i have moved it to the main results section.}

\subsection{SVM and Linear Classifier on MNIST}
\sm{Write it differently. You have done two experiments: (i) learning a two-class SVM classifier and (ii) learning a multi-class linear classifier by cross-entropy minimization. Writing NN training could be misleading.}\hy{fixed}
We consider first the problem of training an two-class SVM classifier and a multi-class linear classifier using features extracted from MNIST. A small 5 layer neural network (2 convolutional layers, 1 dropout layer and 2 fully connected layers) was first trained to a 97\% accuracy, with the penultimate layer having 50 features. We consider the problem of training an SVM on these features for two specific classes. We also consider the problem of retraining the final layer of the neural network for classification, which is equivalent to a linear classifier. Our goal is to minimize the corresponding losses as quickly as possible using LMD. Let us denote the neural network that takes an image and outputs the corresponding 50 features as $\phi:[0,1]^{28\times 28} \rightarrow \R^{50}$. This will work as a feature extractor, on which we will train our SVMs and linear classifiers. \sm{This is a very informal way of writing for a paper.} \hy{hopefully better}

\sm{Explain the architecture through equations. This reads super informal.} \hy{changed}

\subsubsection{SVM}\label{sec:SVM}
Our objective is to train a support vector machine (\ac{SVM}) on the 50 extracted features to classify two classes of digits, namely 4 and 9. Given feature vectors $\phi_i \in \R^d$ and target labels $y_i \in \{\pm 1\}$, an SVM consists of a weight vector $\mathbf{w} \in \R^d$ and bias scalar $b \in \R$. The output of the SVM for a given feature vector is $\mathbf{w}^\top \phi_i + b$, and the aim is to find $\mathbf{w}$ and $b$ such that the prediction $\sign(\mathbf{w}^\top \phi_i + b)$ matches the target $y_i$ for most samples. The hinge loss formulation of the problem is as follows, where $C>0$ is some positive constant \cite{bishop2006SVMBook}:%
\sm{The notation is inconsistent. From the beginning, you have used $x$ to denote the optimization variable. Here, $x$ suddenly denotes something else. You can use $\theta$ for the optimization variable from the beginning, and later on use the same $\theta$ to denote the classifier parameters.}\hy{should be consistent now}%
%
\begin{equation}
    \min_{\mathbf{w},b} \frac{1}{2}\mathbf{w}^\top \mathbf{w} + C \sum_{i} \max(0, 1- y_i(\mathbf{w}^\top \phi_i + b)).
\end{equation}%
The function class that we wish to learn to optimize for is thus
\begin{equation}
\mathcal{F} = \left\{f_{\mathcal{I}}({\mathbf{w},b})=\frac{1}{2}\mathbf{w}^\top \mathbf{w} + C \sum_{i \in \mathcal{I}} \max(0, 1- y_i(\mathbf{w}^\top \phi_i + b))\right\},
\end{equation}
where each instance of $f$ depends on the set of feature-target pairs, indexed by $\mathcal{I}$. We use $C=1$ in our example.  For each training iteration, $\mathcal{I}$ was sampled as a subset of 1000 feature-target pairs from the combined 4 and 9 classes of MNIST, giving us a target function $f_{\mathcal{I}}(\mathbf{w},b) \in \mathcal{F}$. A batch of 2000 initializations $(\mathbf{w},b)$ was then sampled according to a standard normal distribution $\mathbb{P}_{(\mathbf{w},b)|f} = N(0, I_{50+1})$. Subsets from the training fold were used for training LMD, and subsets from the test fold to test LMD. 



\cref{fig:svm_comb} shows the evolution of the hinge loss and SVM accuracy of the LMD method, compared with GD and Adam. We can see that adaptive LMD and LMD with sufficiently large step-size both outperform GD and Adam. In particular, considering LMD with step-size multiplier 2, we can see accelerated convergence after around 10 iterations. One possible interpretation is that the network is learning more about the geometry near the minima, which is why we do not see this increased convergence for smaller step-sizes. The LMD method with approximate backwards map is much more stable in this case, even if it performs slightly worse than LMD with the one-layer NN-based mirror potential as in \cref{fig:svmmdnet}.


\begin{figure}
    \centering
    \includegraphics[width=\textwidth]{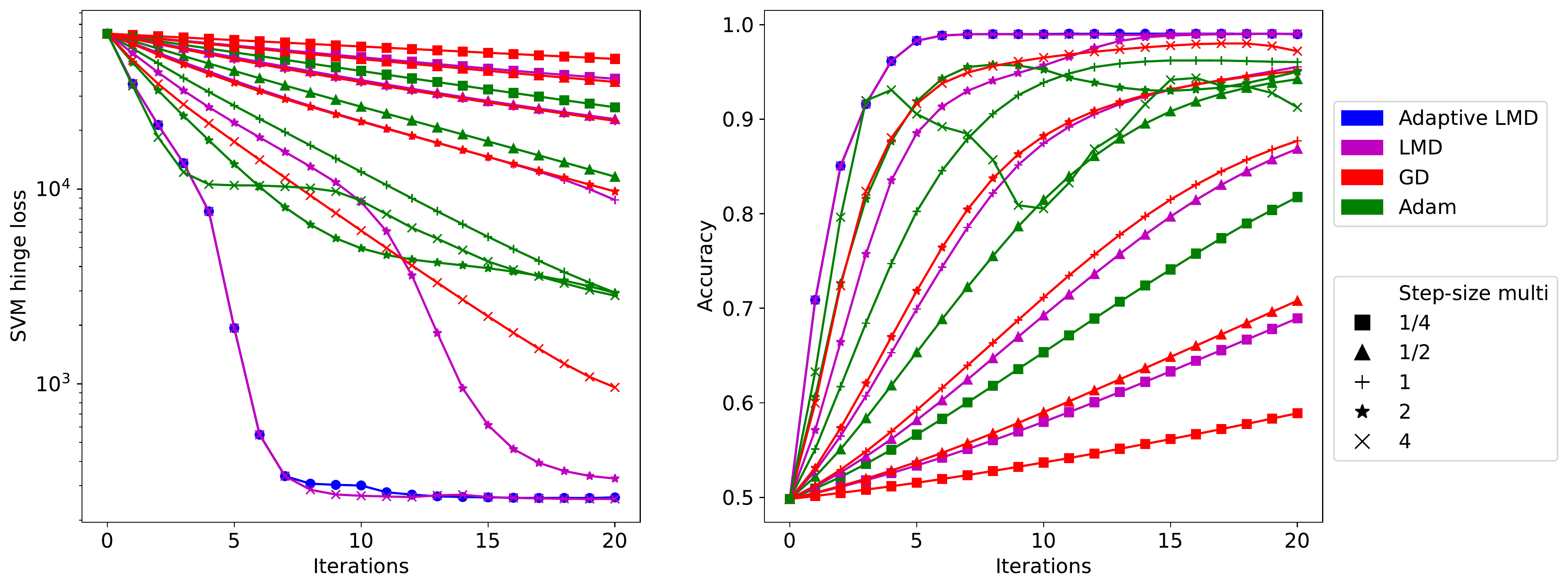}
    \caption{Plot of the SVM hinge loss (left) and SVM test accuracy (right) when optimizing from random SVM initializations. The mirror descent significantly outperforms both gradient descent and Adam, and does not exhibit as large of a decrease in accuracy for later iterations.}
    \label{fig:svm_comb}
\end{figure}

\subsubsection{Linear Classifier}
We additionally consider the problem of training a multi-class linear classifier on the MNIST features. We use the same neural network $\phi$ to produce 50 features, and consider the task of training a linear final layer, taking the 50 features and outputting 10 scores corresponding to each of the digits from 0-9.  The task of finding the optimal final layer with the cross entropy loss can be formulated as follows:%
\begin{equation}
    \min_{W \in \R^{50 \times 10}} \mathbb{E}_{(\phi,y) \in \text{features}\times\text{target}}\left[ -\log \frac{\exp(W\phi)_{y}}{\sum_{i=0}^{9} \exp(W\phi)_{i}}\right].
\end{equation}%
The corresponding feature class we wish to learn to optimize for is:
\begin{equation}
    \mathcal{F} = \left\{f_{\mathcal{I}}(W) = \frac{1}{|\mathcal{I}|} \sum_{(\phi, y) \in \mathcal{I}}\left[ -\log \frac{\exp(W\phi)_{y}}{\sum_{i=0}^{9} \exp(W\phi)_{i}}\right]\right\},
\end{equation}
where each instance of $f$ depends on the set of feature-target pairs, indexed by $\mathcal{I}$.  For each training iteration, $\mathcal{I}$ was sampled as a subset of 2000 feature-target pairs from MNIST, giving a target function $f_{\mathcal{I}}(W) \in \mathcal{F}$. A batch of 2000 initializations $W$ was then sampled according to a standard normal distribution $\mathbb{P}_{W|f} = N(0, I_{50\times 50})$ for training. Subsets from the training fold were used for training LMD, and subsets from the test fold to test LMD. 



\cref{fig:nn_comb} shows the evolution of the cross-entropy loss and neural network classification accuracy under our optimization schemes. All of the LMD methods converge quite quickly, and we see that LMD with smaller step-sizes converge faster than larger step-sizes, reflecting a similar phenomenon in gradient descent. We additionally see that for LMD with step-size multiplier 4, the cross entropy loss has a large spike after 10 iterations. This is likely due to the the step-size being too large for the Lipschitz constant of our problem.

\begin{figure}
    \centering
    \includegraphics[width=\textwidth]{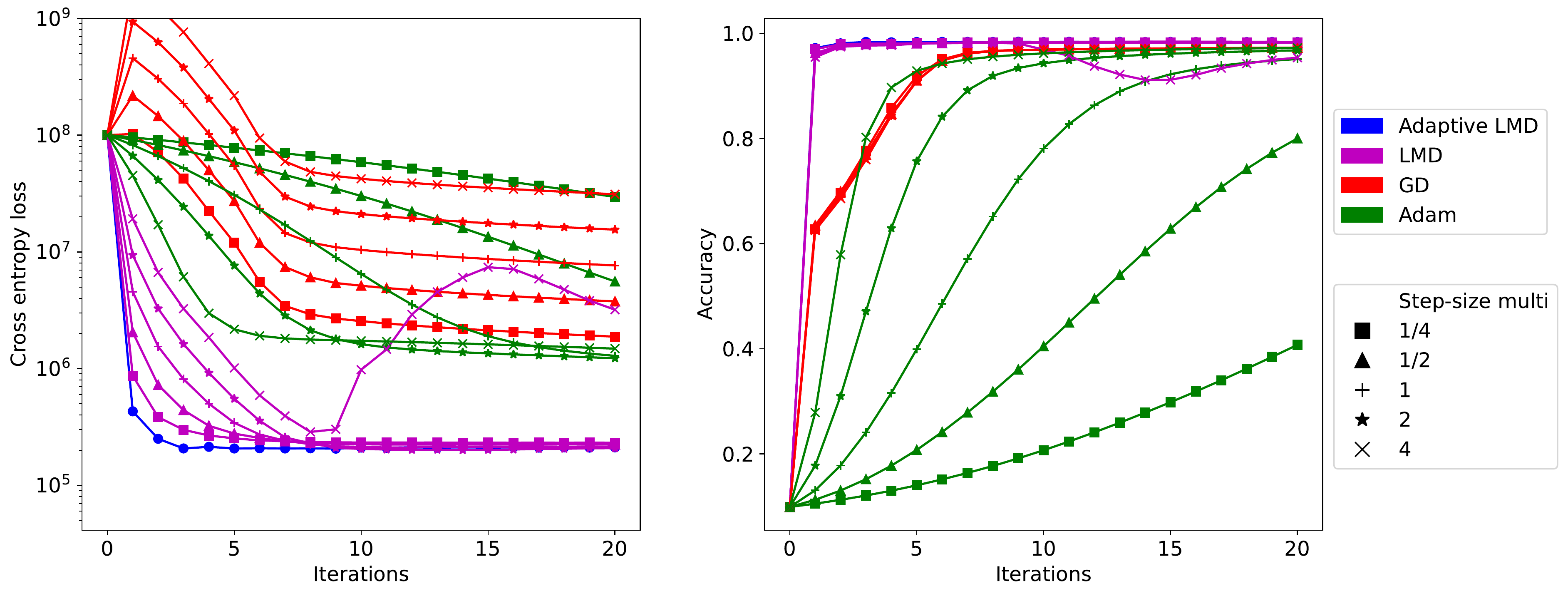}
    \caption{Plots of the linear classifier cross entropy loss (left) and  classification accuracy (right). MD converges significantly faster than both GD and Adam. However, it suffers from stability issues for larger step-sizes, demonstrated by the increase in loss after 10 iterations with step-size multiplier 4. This increase in loss is also reflected in the decrease of accuracy.}
    \label{fig:nn_comb}
\end{figure}
\subsection{Image Denoising}\label{sec:denoising}
\sm{Again the notation changes here.}\hy{changed.}
We further consider the problem of image denoising on the STL-10 image dataset \cite{2011coatesSTL10}. Our goal is to have a fast solver for a single class of variational objectives designed for denoising, rather than devise a state-of-the-art reconstruction approach. As the reconstructions are completely model-driven and do not have a learned component, the quality of the solution will depend completely on the chosen model. 

The denoising problem is to minimize the distance between the reconstructed image with an additional regularization term, which we have chosen to be total variation (TV). The corresponding convex optimization problems can be represented as follows:%
\begin{equation}\label{eq:denoiseTVEqn}
    \min_{x\in \mathcal{X}}\|x-y\|^2_{\mathcal{X}} + \lambda\|\nabla x\|_{1, \mathcal{X}}.
\end{equation}

Here, $\mathcal{X}$ is the space of images from a pixel space $\mathcal{S} \mapsto [0,1]$, $y$ is a noisy image, $\lambda>0$ is a regularization parameter, and the gradient $\nabla x$ is taken over the pixel space. In the case of STL-10, the pixel space is $3 \times 96 \times 96$. The function class we wish to learn to optimize over is thus:%
\begin{equation}\label{eq:denoiseFnClass}
    \mathcal{F} = \left\{f(x) = \|x-y\|^2_{\mathcal{X}} + \lambda\|\nabla x\|_{1, \mathcal{X}}\ : \text{noisy images } y \right\}.
\end{equation}

In our experiments, $y$ was taken to have 5\% random additive Gaussian noise over each color channel, and the initializations $x$ were taken to be the noisy images $x = y$. We trained the LMD method on the training fold of STL10, and evaluated it on images in the test fold. 

The TV regularization parameter was manually chosen to be $\lambda=0.3$ by visually comparing the reconstructions after running gradient descent for 400 iterations. To parameterize the mirror potentials, we use a convolutional neural network with an ICNN structure, as the data is in 2D (with 3 color channels). We additionally introduce a quadratic term in each layer for added expressiveness. The resulting models are of the following form, where the squaring operator $[\, \cdot\, ]^2$ for a vector is to be taken element-wise, and $\sigma$ is a leaky-ReLU activation function:
\begin{equation}
    z_{i+1} = \sigma \left(W_i^{(z)} z_i + W_i^{(x,l)}x + [W_i^{(x,q)}x]^2 + b_i\right),\quad M(x;\theta) = z_l.
\end{equation}
By clipping the kernel weights $W_i^{(z)}$ to be non-negative, we are able to obtain an input convex convolutional neural network.


\sm{This disclaimer sounds super defensive. Phrase it differently, and set out the goal right at the beginning. State clearly that the variational objective is completely model-driven and there is no learned component in it. Devising a state-of-the-art reconstruction approach is not the goal of the work. The goal is to rather design a fast solver for a `given' optimization problem. The quality of the solution depends entirely on the underlying variational problem.}\hy{edited}

\cref{fig:denoise_recon} and \cref{fig:denoiseVis} show the result of applying the LMD algorithm to the function class of denoising models \eqref{eq:denoiseFnClass}. In general, LMD and adaptive LMD outperform GD and Adam for optimizing the reconstruction loss. Moreover, \cref{fig:denoiseVis} shows that the reconstructed image using LMD is very similar to the ones obtained using Adam, which is a good indicator that LMD indeed solves the corresponding optimization problem efficiently. 

\cref{fig:denoiseVisfwd} shows a pixel-wise ratio between the forward map $\nabla M_\theta(y)$ and noisy image $y$. The outline of the horse demonstrates that $\nabla M_\theta$ learns away from the identity, which should contribute to the accelerated convergence. In particular, we observe that around the edges of the horse, the pixel-wise ratio $\nabla M_\theta(y)/y$ is negative. Intuitively, this corresponds to the MD step performing gradient ascent instead of gradient descent for these pixels. As we are using TV regularization, the gradient descent step aims to create more piecewise linear areas. If we interpret gradient descent as a ``blurring" step, then MD will instead perform a ``sharpening" step, which is more suited around the edges of the horse. 

We additionally consider the effect of changing the noise level, and the ability of LMD to generalize away from the training function class. We keep the LMD mirror maps trained for 5\% additive Gaussian noise, and apply LMD to denoise images from STL-10 with additive Gaussian noise levels up to 20\%. We consider now the PSNR and SSIM of the denoised images compared to a ``true TV reconstruction", which is obtained by optimizing the objective \cref{eq:denoiseTVEqn} to a very high accuracy using gradient descent for 4000 iterations. We compare the iterates with respect to the true TV reconstruction as opposed to the ground truth, as we want to compare the resulting images with the minimum of the corresponding convex objective.


\Cref{tab:psnrssimdenoise} compares the PSNR and SSIM of denoised images obtained using LMD, Adam, and GD, compared against the true TV reconstructions. We apply GD and LMD with five fixed step-sizes ranging from $2.5\times 10^{-3}$ to $4\times 10^{-2}$ up to 20 iterations, and Adam with five learning rates ranging from $1.25\times 10^{-2}$ to $2\times 10^{-1}$ for 20 iterations. We then compare the best PSNR/SSIMs over all step-sizes and iterations for each method, and the best overall step-sizes for the 10th and 20th iteration. 

We see that LMD outperforms both GD and Adam when applied on the trained noise level of 5\% for the trained number of iterations $N=10$, with better SSIM up to 10\% noise as well. LMD also performs well for lower noise levels, which can be attributed to good forward-backward consistency near the true TV reconstruction. However, LMD begins to diverge for larger noise levels. This can be attributed to the increased noise being out of the training distribution, increasing the forward-backward loss and thereby causing instabilities.

\begin{figure}[h]
    \centering
    \includegraphics[width=0.9\textwidth,height=\textheight,keepaspectratio]{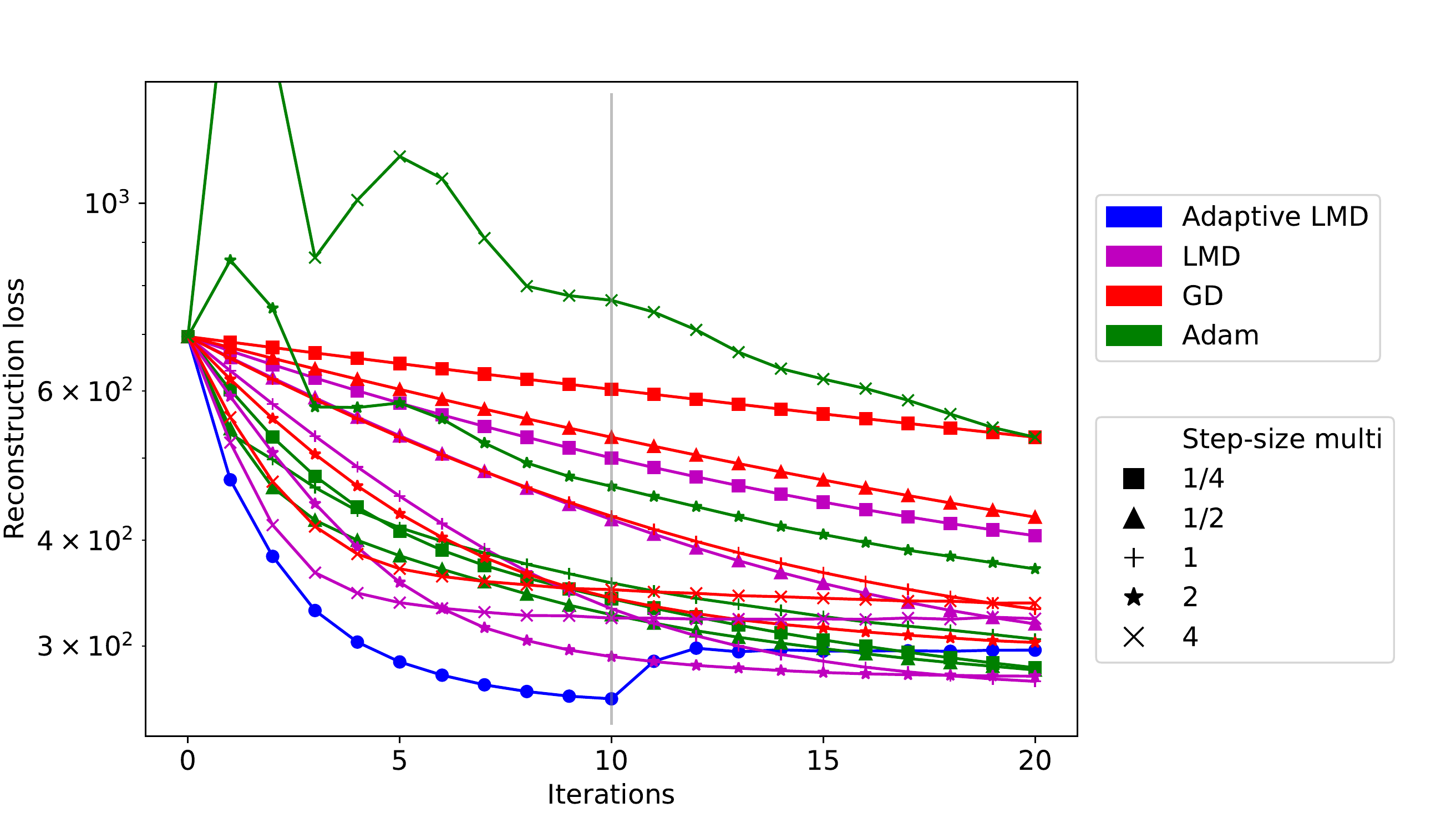}
    \caption{Denoising reconstruction loss. The vertical gray line at iteration 10 indicates the end of the training regime. After this line, the iterates are out-of-distribution for the proposed method.} LMD outperforms both GD and Adam for earlier iterations, however might not reach the minimum due to forward-backward inconsistency. The sharp increase in loss for adaptive LMD after 10 iterations is due to the choice of step-size to extend the trained 10 iterations. 
    \label{fig:denoise_recon}
\end{figure}
\begin{figure}[h]%
    \centering
    \subfloat[\centering Ratio between forward map and noisy image ]{{\includegraphics[height=3.6cm]{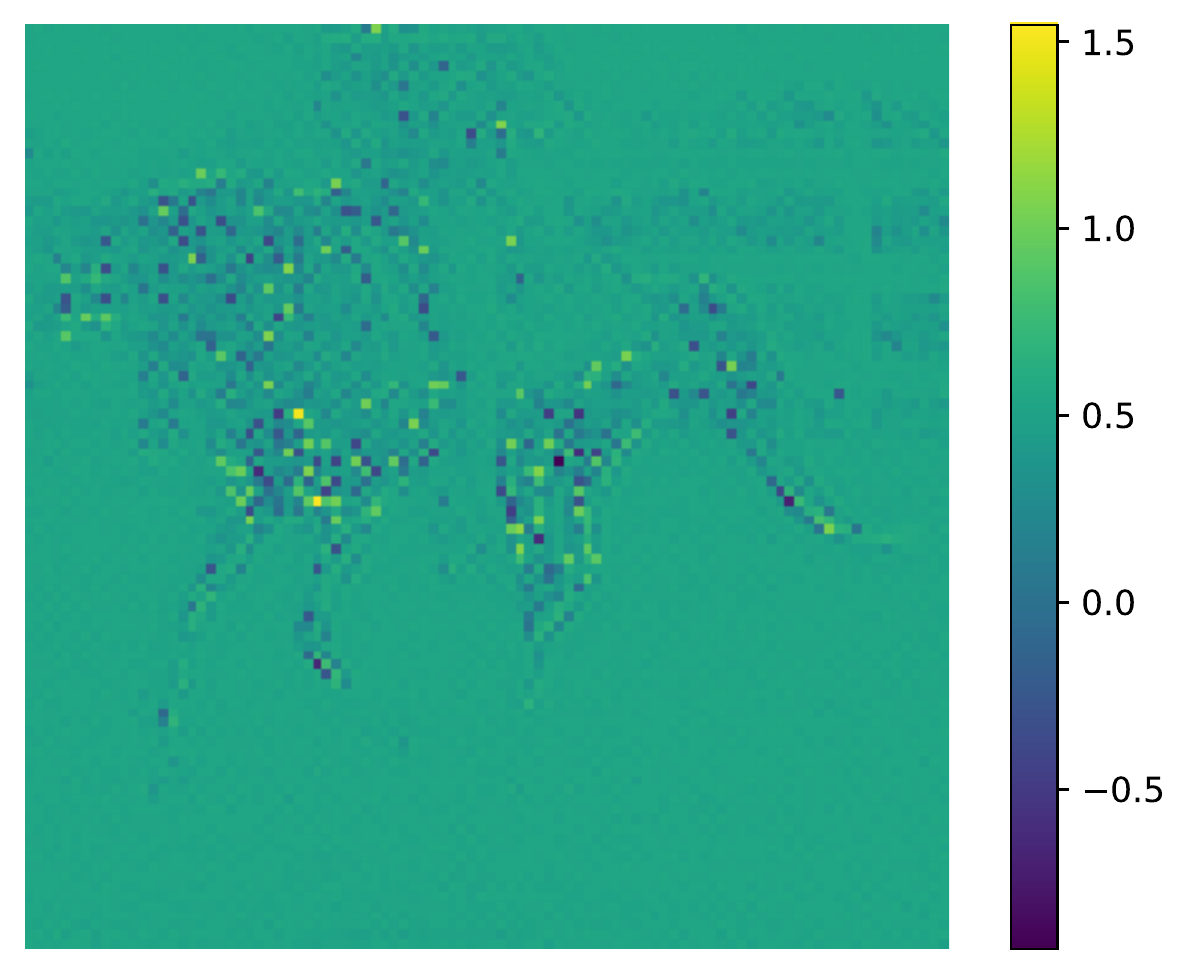}\label{fig:denoiseVisfwd} }}%
    \subfloat[\centering Reconstruction after 3 iterations of adaptive LMD]{{\includegraphics[height=3.6cm]{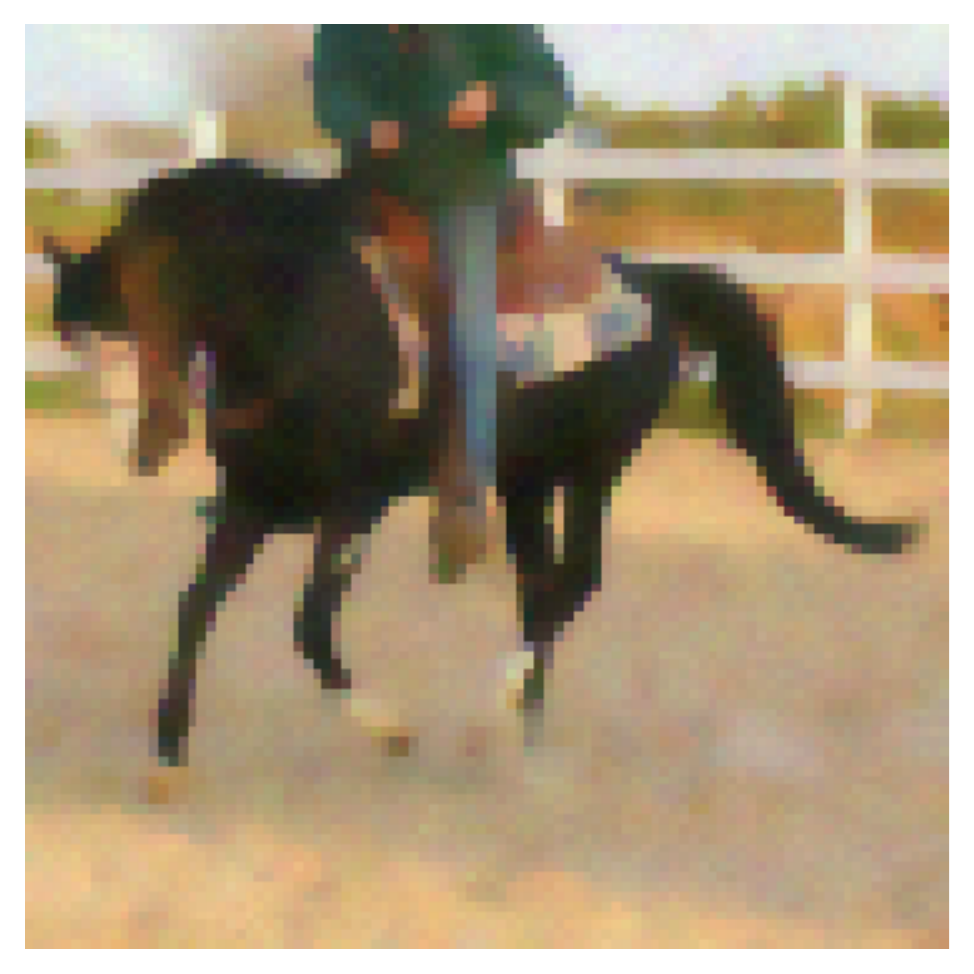}\label{fig:denoiseVismd}}}%
    \subfloat[\centering Reconstruction after 3 iterations of Adam]{{\includegraphics[height=3.6cm]{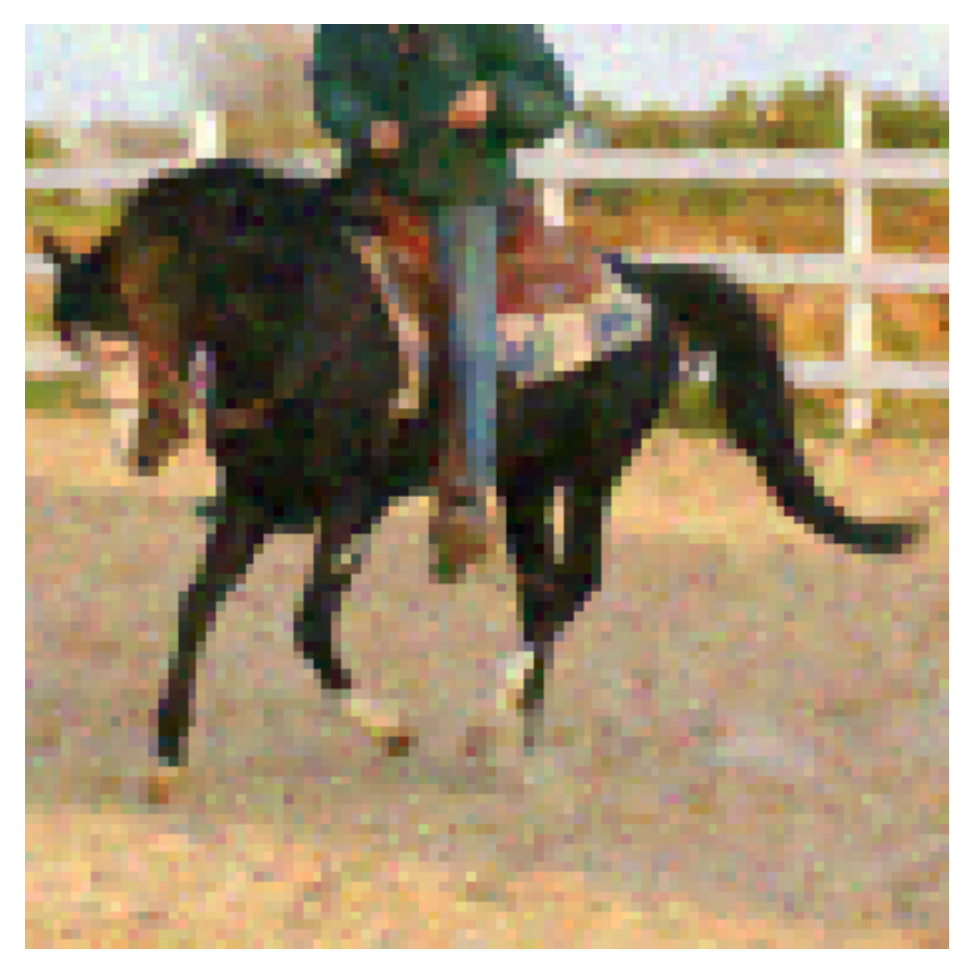}}}%
    \subfloat[\centering Reconstruction after 10 iterations of Adam]{{\includegraphics[height=3.6cm]{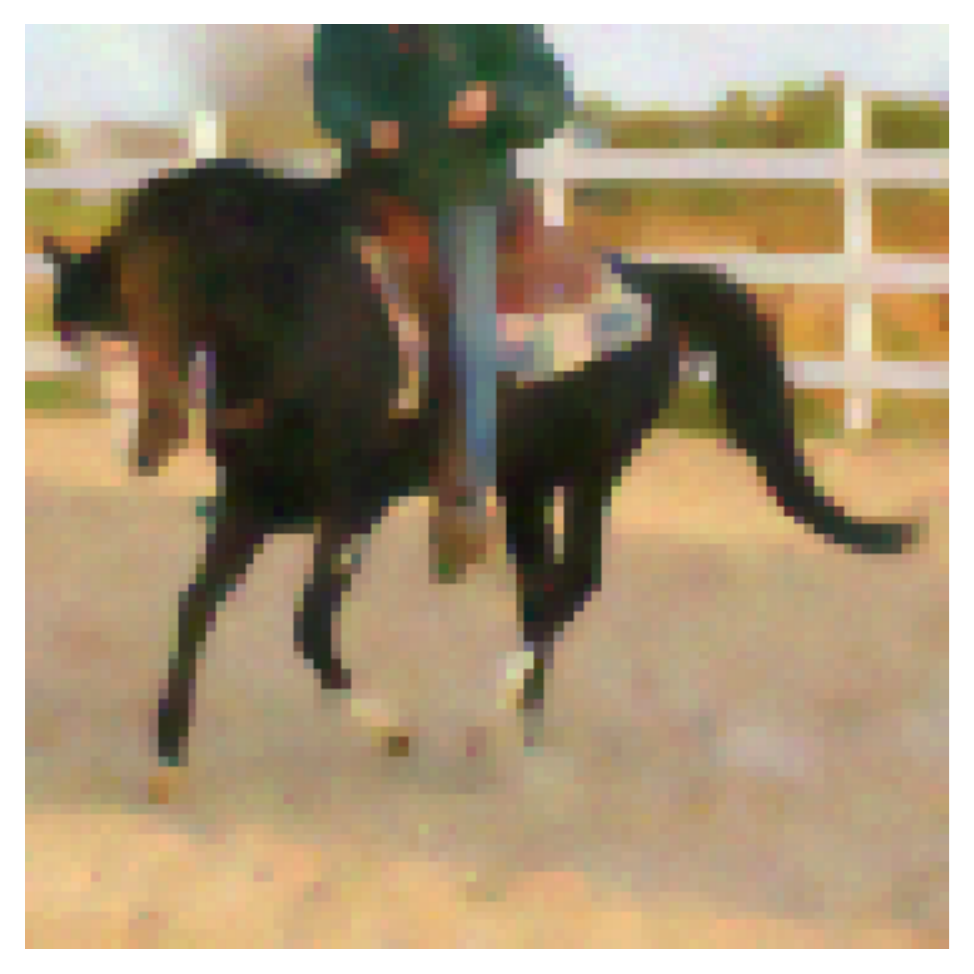}\label{fig:denoiseVisadam}}}%
    \caption{Visualization of outputs when when applying LMD for TV model-based denoising. We can see a faint outline of the horse when taking a pixel-wise ratio between the forward and noisy image indicating a region of interest. LMD allows for much faster convergence compared to Adam here, reaching a comparable reconstruction in only 3 iterations compared to 10 for Adam.}%
    \label{fig:denoiseVis}%
\end{figure}

\begin{table}[h]
\centering
\caption{Table of PSNR and SSIM, compared to the true TV reconstruction. As our goal is to minimize the TV-regularized loss function, we compare with the loss-minimizing image as opposed to the ground truth image. LMD outperforms both GD and Adam when applied for noise levels up to 5\% for the trained $N=10$ iterations, but is unstable for noise levels above 10\%, which are out-of-distribution. Values are taken as the best over five step-sizes.}
\label{tab:psnrssimdenoise}
\resizebox{\textwidth}{!}{%
\begin{tabular}{ll|lll|lll|lll}
\hline
\multicolumn{2}{l|}{\multirow{2}{*}{Gaussian Noise \%}} & \multicolumn{3}{c|}{Best} & \multicolumn{3}{c|}{Iteration 10} & \multicolumn{3}{c}{Iteration 20} \\ \cline{3-11} 
\multicolumn{2}{l|}{}        & GD    & Adam           & LMD            & GD    & Adam           & LMD            & GD    & Adam           & LMD            \\ \hline
\multirow{6}{*}{PSNR} & 1 & 30.91 & 34.13          & \textbf{34.25} & 27.92 & 31.04          & \textbf{33.27} & 30.91 & \textbf{34.03} & 32.88          \\
                      & 2 & 30.93 & 34.06          & \textbf{34.21} & 27.90 & 30.86          & \textbf{33.21} & 30.93 & \textbf{34.00} & 32.87          \\
                      & 5 & 31.09 & 33.36          & \textbf{34.22} & 27.73 & 29.91          & \textbf{32.92} & 31.09 & \textbf{33.44} & 33.11          \\
                      & 10 & 31.08 & \textbf{32.59} & 28.21          & 26.45 & \textbf{29.00} & 27.88          & 31.08 & \textbf{32.40} & 25.92          \\
                      & 15 & 29.68 & \textbf{32.56} & 21.39          & 23.65 & \textbf{28.89} & 19.84          & 29.68 & \textbf{32.30} & 13.25          \\
                      & 20 & 28.96 & \textbf{33.68} & 20.12          & 22.97 & \textbf{30.32} & 10.94          & 28.96 & \textbf{33.37} & -21.09         \\ \hline
\multirow{6}{*}{SSIM} & 1 & 0.905 & 0.960          & \textbf{0.963} & 0.862 & 0.914          & \textbf{0.956} & 0.905 & 0.956          & \textbf{0.961} \\
                      & 2 & 0.905 & 0.955          & \textbf{0.963} & 0.858 & 0.908          & \textbf{0.955} & 0.905 & 0.951          & \textbf{0.961} \\
                      & 5 & 0.898 & 0.935          & \textbf{0.962} & 0.857 & 0.880          & \textbf{0.950} & 0.898 & 0.932          & \textbf{0.961} \\
                      & 10 & 0.893 & 0.907          & \textbf{0.950} & 0.817 & 0.831          & \textbf{0.908} & 0.893 & 0.902          & \textbf{0.950} \\
                      & 15 & 0.876 & \textbf{0.893} & 0.849          & 0.698 & \textbf{0.799} & 0.689          & 0.876 & \textbf{0.889} & 0.849          \\
                      & 20 & 0.850 & \textbf{0.917} & 0.887          & 0.662 & \textbf{0.841} & 0.772          & 0.850 & \textbf{0.915} & 0.878          \\ \hline
\end{tabular}%
}
\end{table}

\subsection{Image Inpainting}\label{sec:inpainting}
We additionally consider the problem of image inpainting with added noise on STL10, in a similar setting to image denoising. 20\% of the pixels in the image were randomly chosen to be zero to create a fixed mask $Z$, and 5\% Gaussian noise was added to the masked images to create noisy masked images $y$. The inpainting problem is to minimize the distance between the masked reconstructed image and the noisy masked image, including TV regularization. The corresponding convex optimization problem is%
\begin{equation}
        \min_{x\in \mathcal{X}}\, \|Z \circ (x-y)\|^2_{\mathcal{X}}\, +\, \lambda\|\nabla x\|_{1, \mathcal{X}},
\end{equation}
where $Z$ denotes the masking map $\mathcal{S} \mapsto \{0,1\}^d$, and the image difference $x-y$ is taken pixel-wise. The corresponding function class that we wish to learn to optimize over is:
\begin{equation}\label{eq:inpaintingFnClass}
    \mathcal{F} = \left\{f(x) = \|Z \circ (x-y)\|^2_{\mathcal{X}}\, +\, \lambda\|\nabla x\|_{1, \mathcal{X}}\ : \text{noisy masked images } y \right\}.
\end{equation}
    
The initializations $x$ were taken to be the noisy masked images $x=y$. We trained the LMD method on the training fold of STL10, and evaluated it on images in the test fold. The TV regularization parameter was chosen to be $\lambda=0.3$ as in the denoising case, and the mirror potentials are parameterized with a convolutional neural network similar to that used in the denoising experiment. We trained the LMD method on the training fold of STL10, and evaluated it on images in the test fold. 

\cref{fig:inpaint_recon} shows the loss evolution of applying the LMD algorithm to the function class of inpainting models \eqref{eq:inpaintingFnClass}. LMD with sufficiently large step-size outperforms GD and Adam, however having too small of a step-size can lead to instability. We can also clearly see the effect of approximating our backward maps, as some of the LMD methods result in asymptotic reconstruction loss that is higher than a minimum. Nonetheless, adaptive LMD results in the best convergence out of the tested methods.


\cref{fig:inpaintVis} provides a visualization of the resulting iterations. \cref{fig:inpaintVisfwd} plots the ratio between the forward mapped masked image $\nabla M_\theta(y)$ and masked image $y$, with clipped values to prevent blowup in the plot. We can again see a faint outline of the horse indicating a region of interest, with some speckling due to the image mask. \cref{fig:inpaintVismd} is a plot of the result after 20 iterations of adaptive LMD, and it is qualitatively quite similar to the result after 20 iterations of Adam, demonstrating the feasibility of LMD as a solver for model-based reconstruction.

\begin{table}[]
\centering
\caption{Hyper-parameters for the problem classes considered.}
\label{tab:hyperparams}
\resizebox{\textwidth}{!}{%
\begin{tabular}{@{}lcccc@{}}
\toprule
                     & \multicolumn{1}{l}{SVM} & \multicolumn{1}{l}{Linear Classifier} & \multicolumn{1}{l}{Denoising} & \multicolumn{1}{l}{Inpainting} \\ \midrule
Batch size                       & 2000      & 2000      & 10     & 10     \\
Epochs                           & 10,000    & 10,000    & 1300   & 1100   \\ \midrule
                                 & \multicolumn{4}{c}{All}                 \\ \midrule
ICNN training parameters (Adam) & \multicolumn{4}{c}{$\alpha=10^{-5}, \beta = (0.9,0.99)$}                                                                         \\
Learned iterations $N$           & \multicolumn{4}{c}{10}                  \\
Learned step-size initialization & \multicolumn{4}{c}{$10^{-2}$}           \\
Learned step-size range          & \multicolumn{4}{c}{$(10^{-3},10^{-1})$} \\
Testing base step-size (LMD,GD) & \multicolumn{4}{c}{$10^{-2}$}           \\
Testing base step-size (Adam)    & \multicolumn{4}{c}{$5\times 10^{-2}$}   \\ \bottomrule
\end{tabular}%
}
\end{table}

\begin{figure}%
    \centering
    \includegraphics[width=.9\textwidth,keepaspectratio]{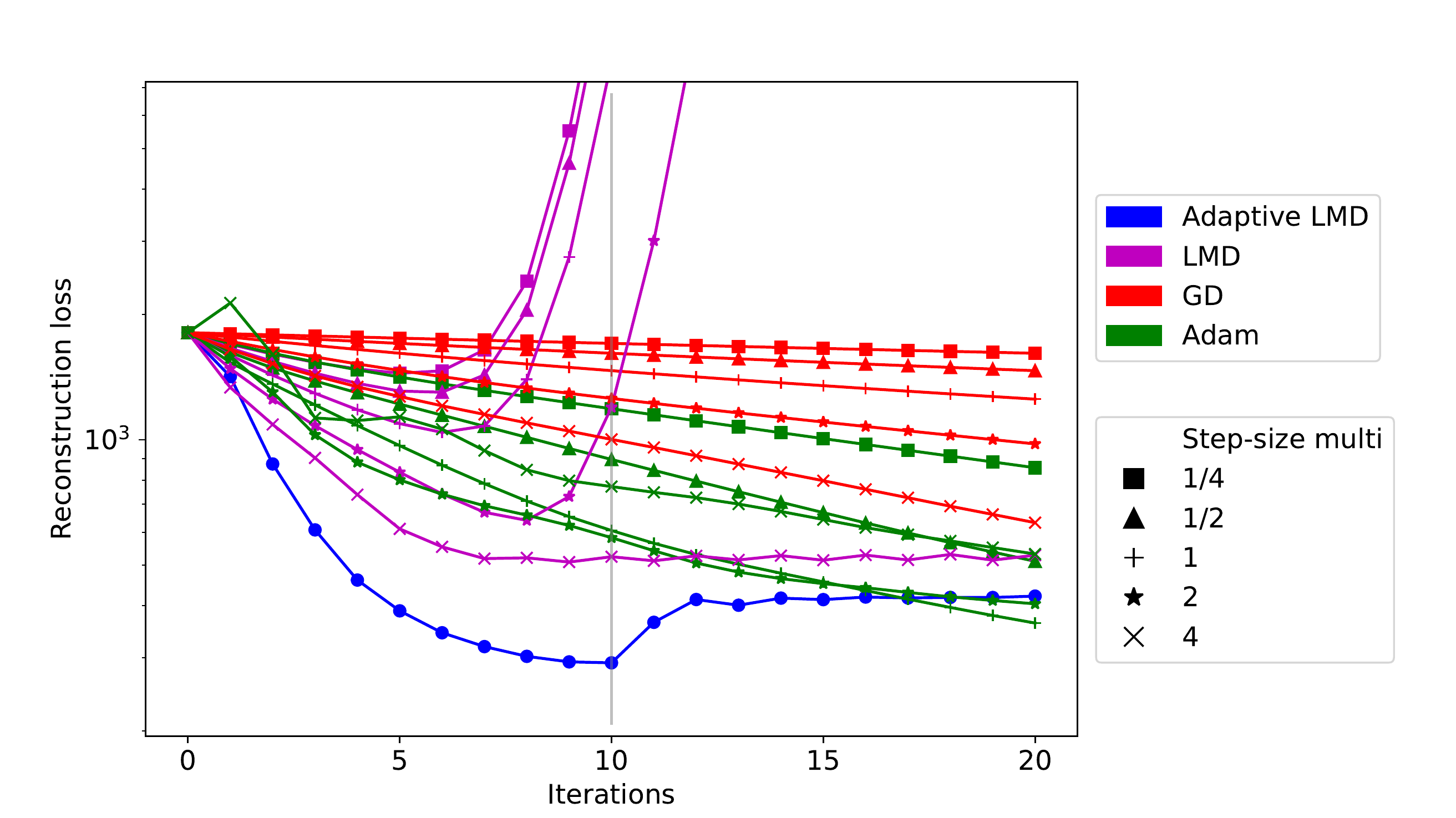}
    \caption{Inpainting reconstruction loss. The vertical gray line at iteration 10 indicates the end of the training regime. LMD outperforms both GD and Adam, however suffers from instability when the step-size is small, as remarked in \cref{rem:smallSS}. The increase in loss after 10 iterations for adaptive LMD is due to the choice of step-size to extend the trained 10 iterations. }
    \label{fig:inpaint_recon}
\end{figure}%
\begin{figure}%
    \centering
    \subfloat[\centering Ratio between forward map and masked image ]{{\includegraphics[height=3.6cm]{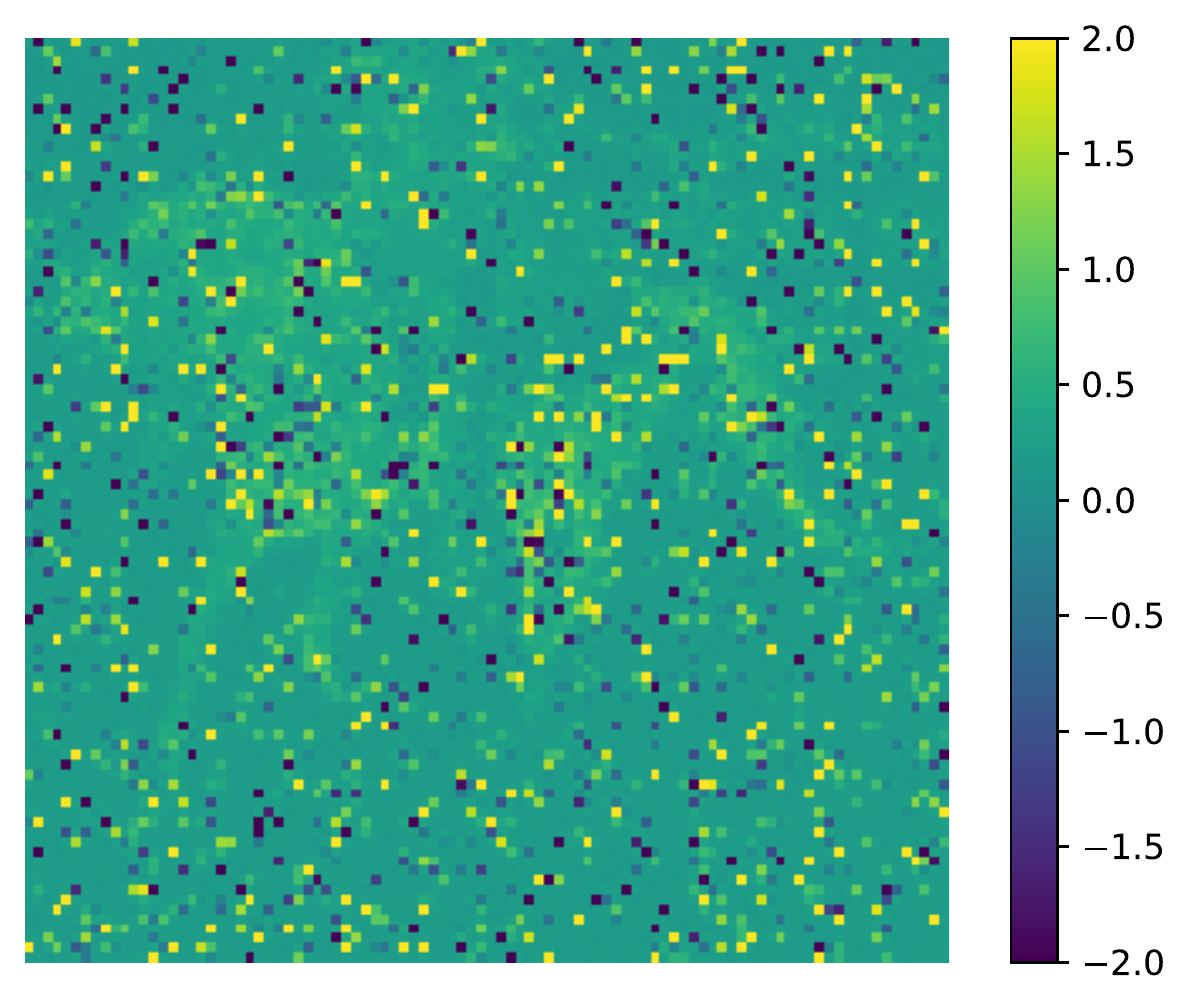}\label{fig:inpaintVisfwd} }}%
    \subfloat[\centering Reconstruction after 10 iterations of adaptive LMD]{{\includegraphics[height=3.6cm]{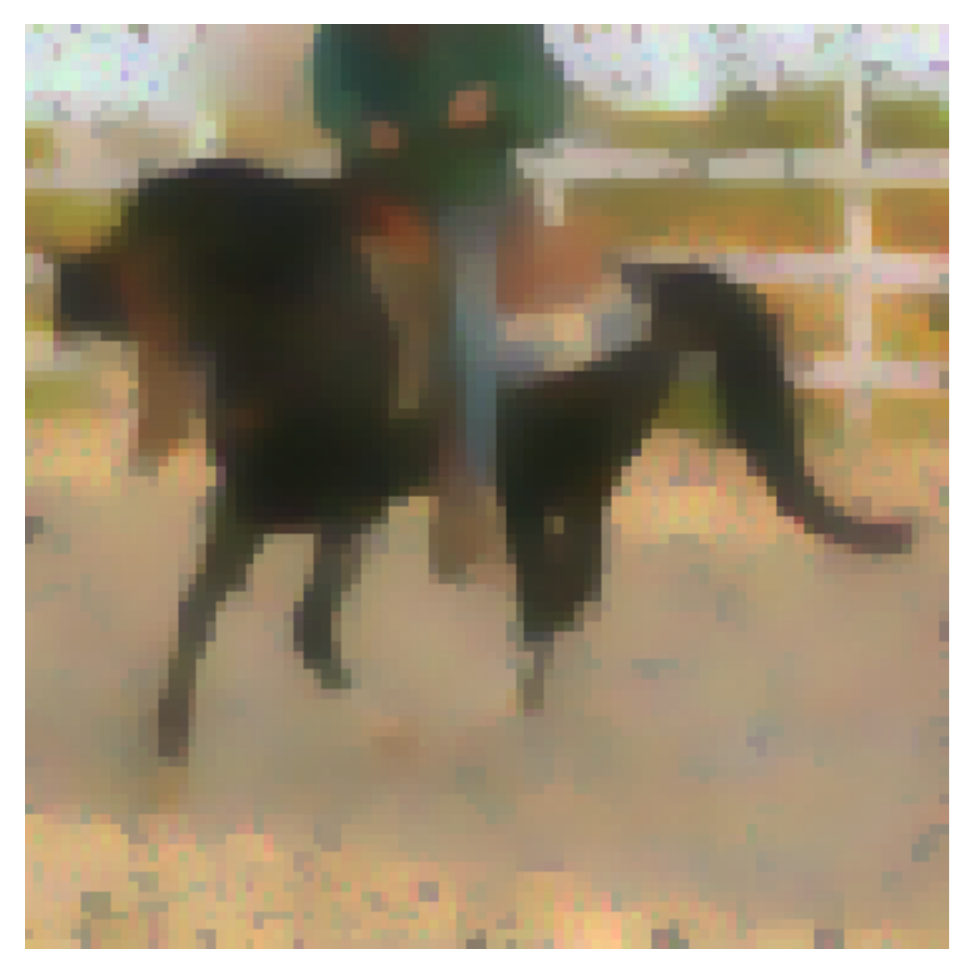}\label{fig:inpaintVismd}}}%
    \subfloat[\centering Reconstruction after 10 iterations of Adam]{{\includegraphics[height=3.6cm]{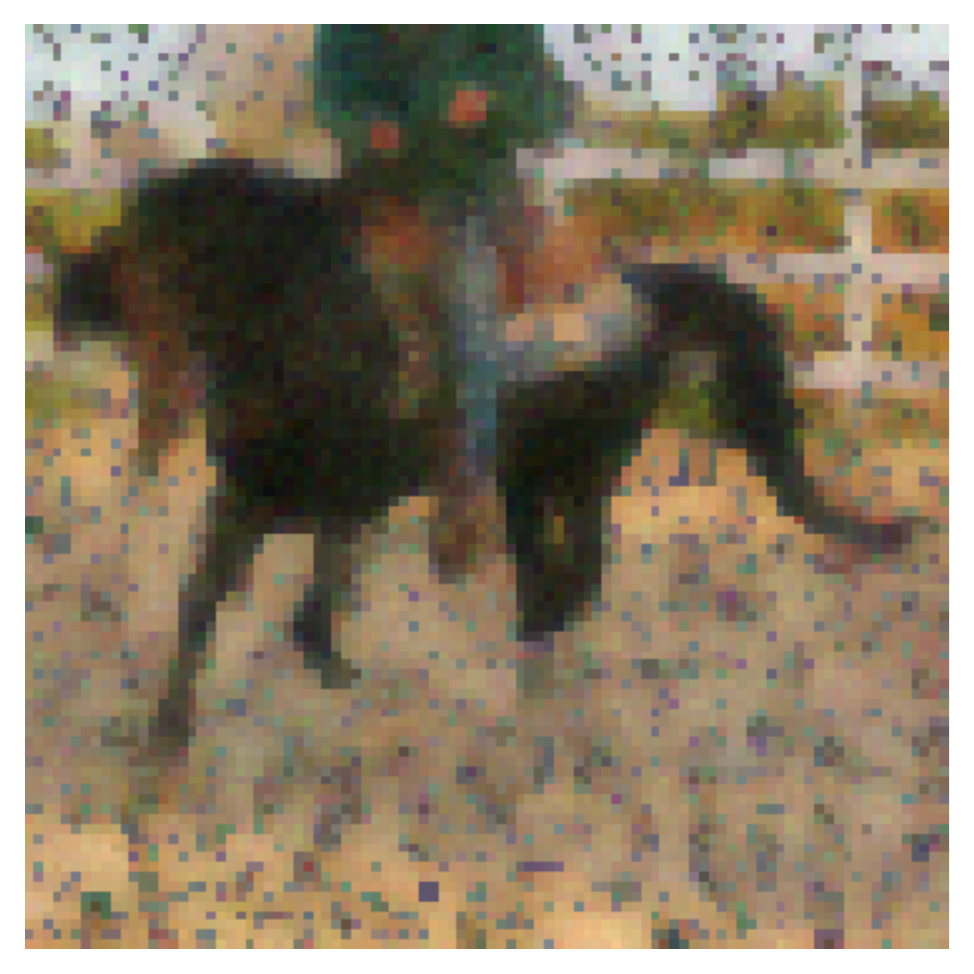}}}%
    \subfloat[\centering Reconstruction after 20 iterations of Adam]{{\includegraphics[height=3.6cm]{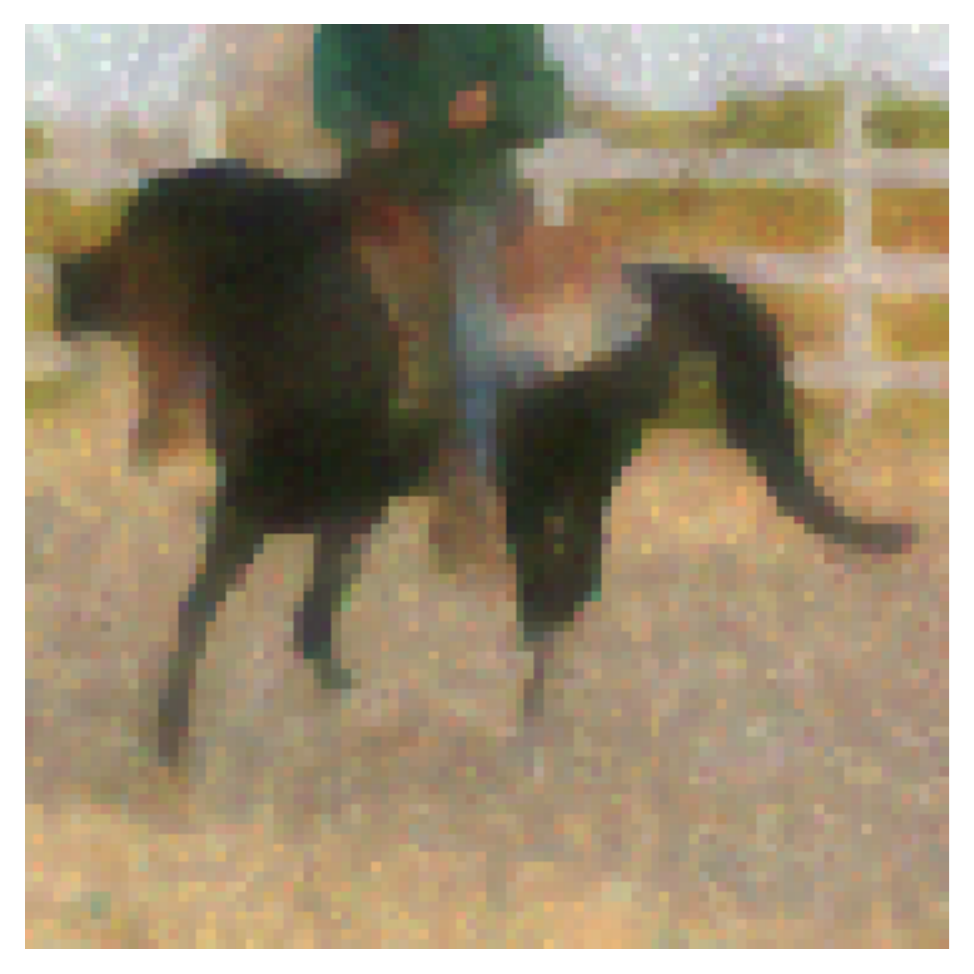}\label{fig:inpaintVisadam}}}%
    \caption{Visualization of some LMD on TV model-based inpainting. While a faint outline of the horse is visible, it is not as clear as in \cref{fig:denoiseVis} with speckling due to the zeroing mask. LMD is able to reach a reasonable reconstruction in fewer iterations compared to Adam. While the LMD reconstruction has artifacts around the edges, the Adam reconstruction is generally noisy.}
    \label{fig:inpaintVis}%
\end{figure}%

\subsection{Effect of Regularization Parameter}
We now turn to studying the effect of the regularization parameters used to enforce consistency of the forward and backward mirror maps\sm{associated with training. Mention that clearly}\hy{fixed}. The regularization parameter $s_k = s_{\text{epoch}}$ as in \eqref{eq:trainingLoss} was initialized as 1, and subsequently multiplied by 1.05 every 50 epochs. 

Under the assumption that the model is trained well for each regularization parameter, the training loss gives a perspective into the trade-off between the loss and the forward-backward consistency of the learned mirror maps. Informally, the model will try to learn a one-shot method similar to an end-to-end encoder-decoder model. Increasing the forward-backward regularization parameter $s_{\text{epoch}}$ reduces this one-shot effect, and encourages a proper optimization scheme to emerge. Therefore, it is natural that the objective loss will increase as the forward-backward loss decreases. This effect can be seen in \cref{fig:progSVM}, where the objective loss starts very low but then increases as the forward-backward error decreases. This could be interpreted as the LMD learning a single good point, then switching to learning how to optimize to a good point. In addition to encouraging a proper optimization scheme, increasing the forward-backward regularization parameter has the added effect of encouraging the forward-backward loss to continue decreasing. This can be seen in \cref{fig:progInpaint}, where the objective loss also decreases before increasing again.%


\sm{You need to explain the findings of the experiments in the main text, while making a reference to the corresponding image/plot. Additionally, the figure captions should give a concise description of the setting and the key takeaways.}\hy{attempted.}


\begin{figure}
\centering
\begin{minipage}[t]{.48\textwidth}
  \centering
  \includegraphics[width=\linewidth]{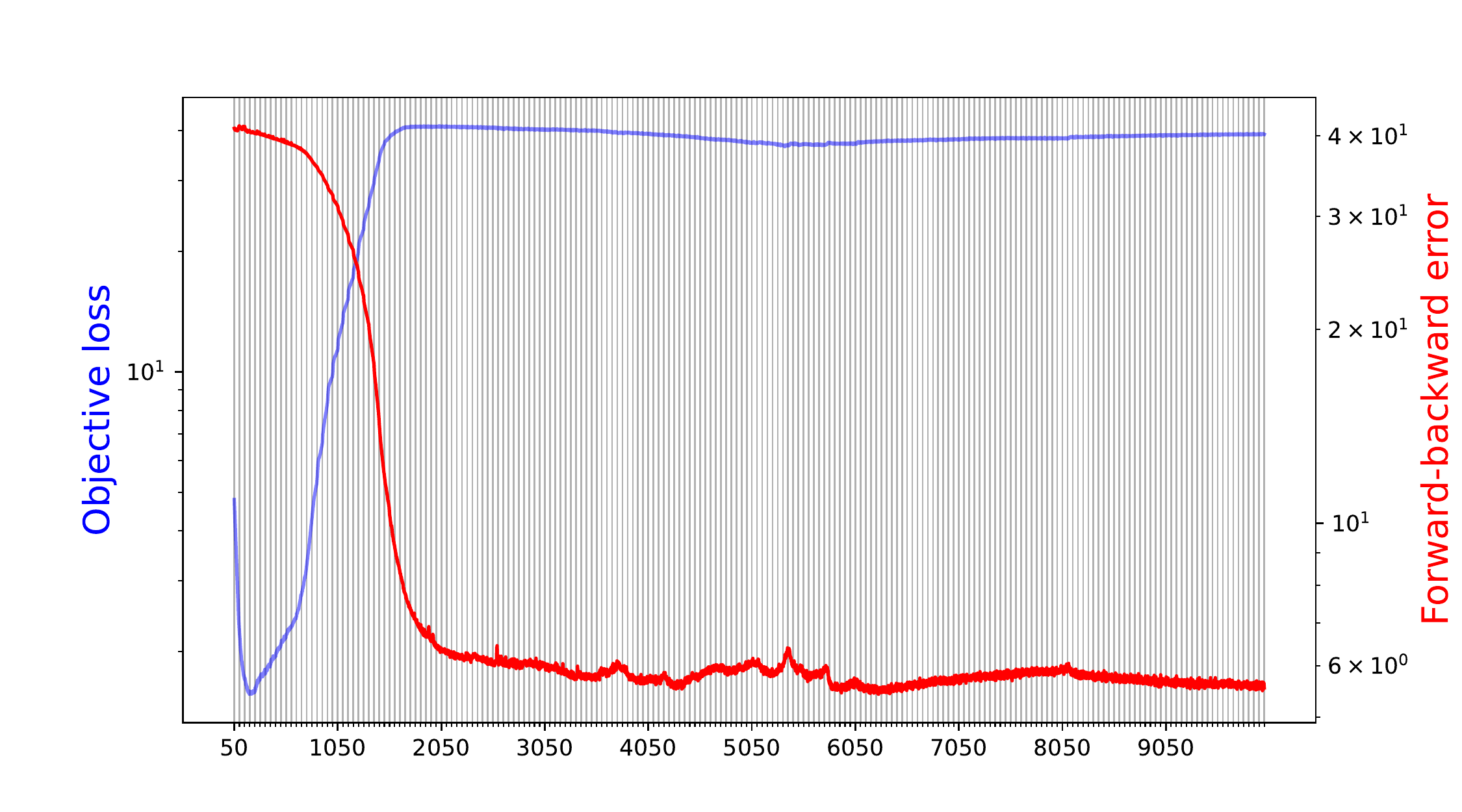}
  \captionof{figure}{Training loss and forward-backward consistency loss when training an SVM, plotted against training epochs. We can see clearly the tradeoff between the loss and forward-backward loss at the earlier iterations. Each vertical grey line corresponds to an epoch where the forward-backward loss regularization is increased.}
 \label{fig:progSVM}
\end{minipage} \quad \begin{minipage}[t]{.48\textwidth}
  \includegraphics[width=\linewidth]{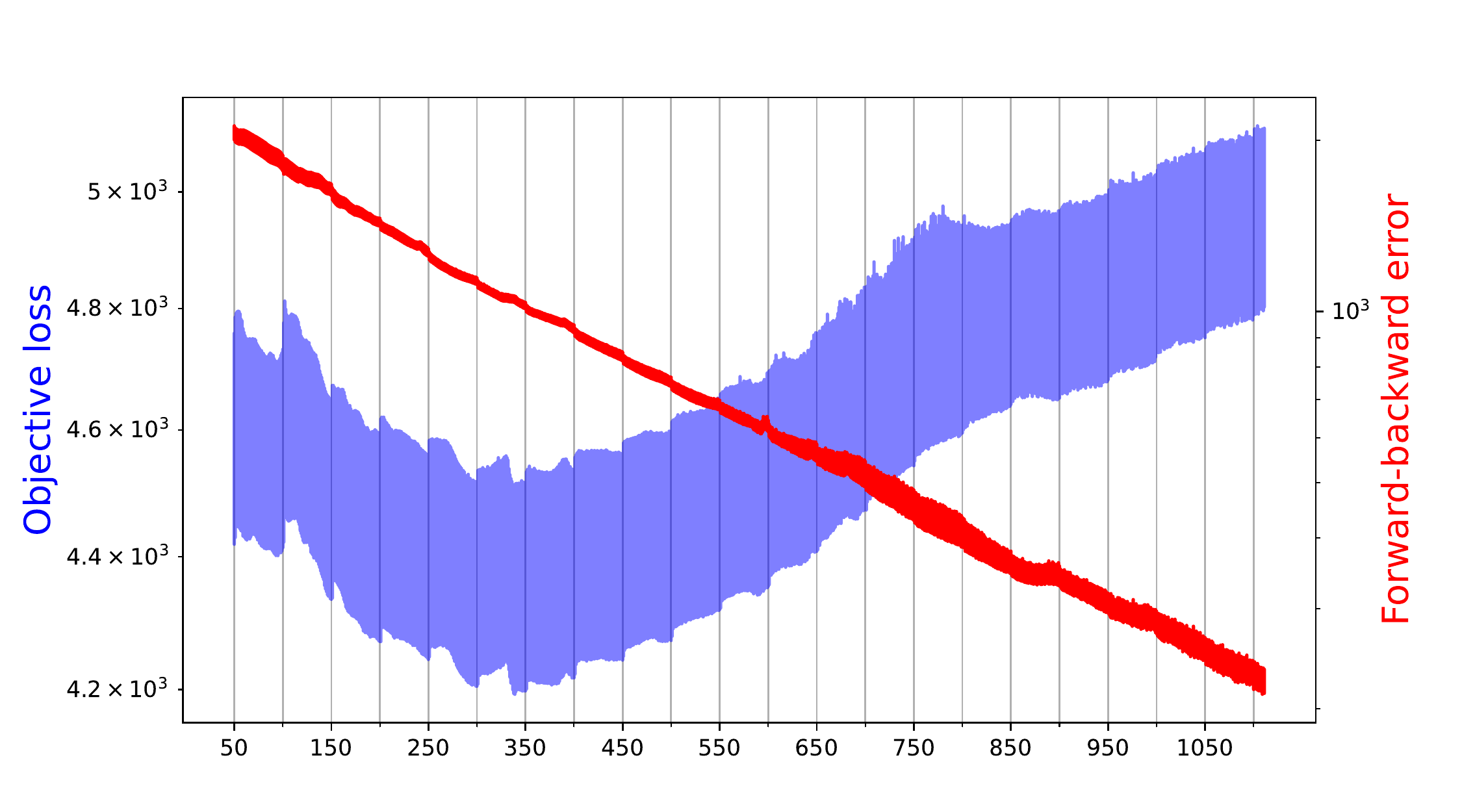}
  \captionof{figure}{Training loss and forward-backward consistency loss when training inpainting on STL10, plotted against training epochs. We can see the effect of increasing the forward-backward regularization parameter as the forward-backward loss continues to decrease along the iterations, while the loss begins to increase. }
  \label{fig:progInpaint}
\end{minipage}
\end{figure}

\subsection{Ablation Study}
In this section, we will compare the effect of various design choices on LMD. In particular, we will consider (i) the effect of the number of training iterations $N$, (ii) the effect of not enforcing the forward-backward consistency by setting $s_k = 0$, and (iii) a further comparison against GD with learned step-sizes (LGD). In particular, the first experiment will be LMD trained with $N=2$. The latter experiment is equivalent to our LMD with both mirror maps fixed to be the identity. We will compare these three experiments on the inpainting setting as in \Cref{sec:inpainting}. \Cref{fig:ablation} compares the forward-backward inconsistency and the loss for these three experiments with LMD trained for inpainting, detailed in \Cref{sec:inpainting}. For each of these methods, we choose to extend the learned step-sizes by a constant, up to 20 iterations. 

For experiment (i), decreasing the number of training iterations $N$ severely impacts the forward-backward inconsistency. Moreover, the number of training iterations is insufficient to be close to the minimum of the problem. These problems coupled together lead to the loss converging to a poor value, or diverging depending on the step-size extension.

For experiment (ii), setting $s_k=0$ in \cref{eq:trainingLoss} and not enforcing forward-backward consistency results in high forward-backward loss. Nonetheless, the loss rapidly decreases in the first couple iterations, faster than LMD. This is consistent with the view that the pair of mirror potentials acts as an encoder-decoder network, rapidly attaining close to the minimum. Due to the higher forward-backward loss, this method has looser bounds on the convergence, resulting in the increase in reconstruction loss in the later iterations compared to LMD.

For experiment (iii), learning the step-sizes for GD directly results in significantly worse performance compared to LMD. This can be attributed to LMD learning the direction of descent via the mirror maps, in addition to the speed of descent given by the learned step-sizes. This demonstrates that a better direction than steepest descent exists, can be learned by LMD and results in faster convergence rates.

These experiments demonstrate the effect of the variables of LMD. In particular, we show that a sufficient number of training iterates is required to maintain longer term convergence, and that enforcing forward-backward consistency sacrifices some early-iterate convergence rate for better stability for longer iterations. Moreover, the LGD experiment shows that learning a direction via the mirror maps in addition to the speed of convergence allows for faster convergence.

\begin{figure}[h]%
    \centering
    \subfloat[\centering Evolution of forward-backward loss. ]{{\includegraphics[clip, trim=25 0 190 0,  height=6.1cm]{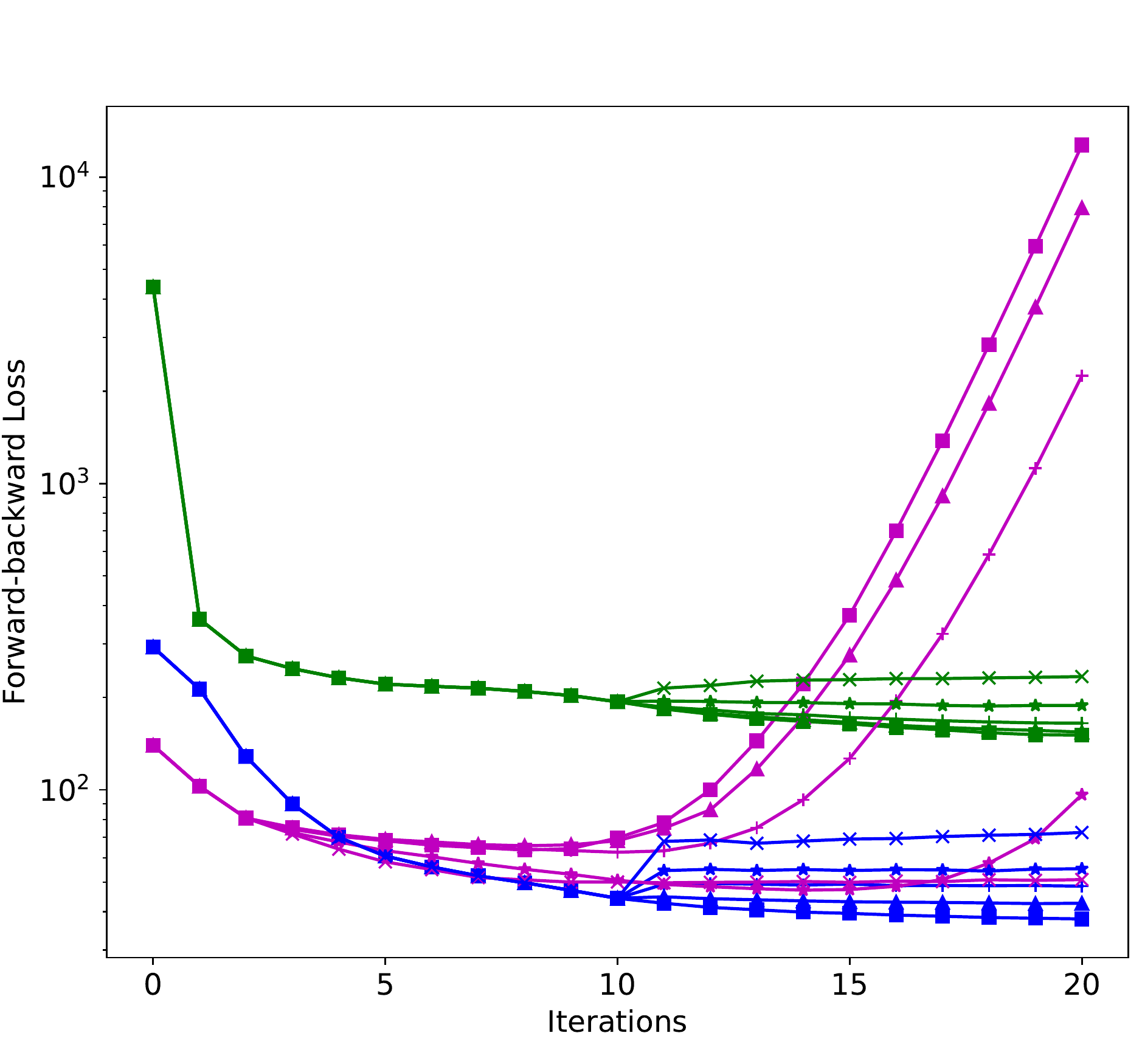}\label{fig:altfwdbwd} }}%
    \subfloat[\centering Evolution of inpainting reconstruction loss. ]{{\includegraphics[clip, trim=25 0 0 0, height=6.1cm]{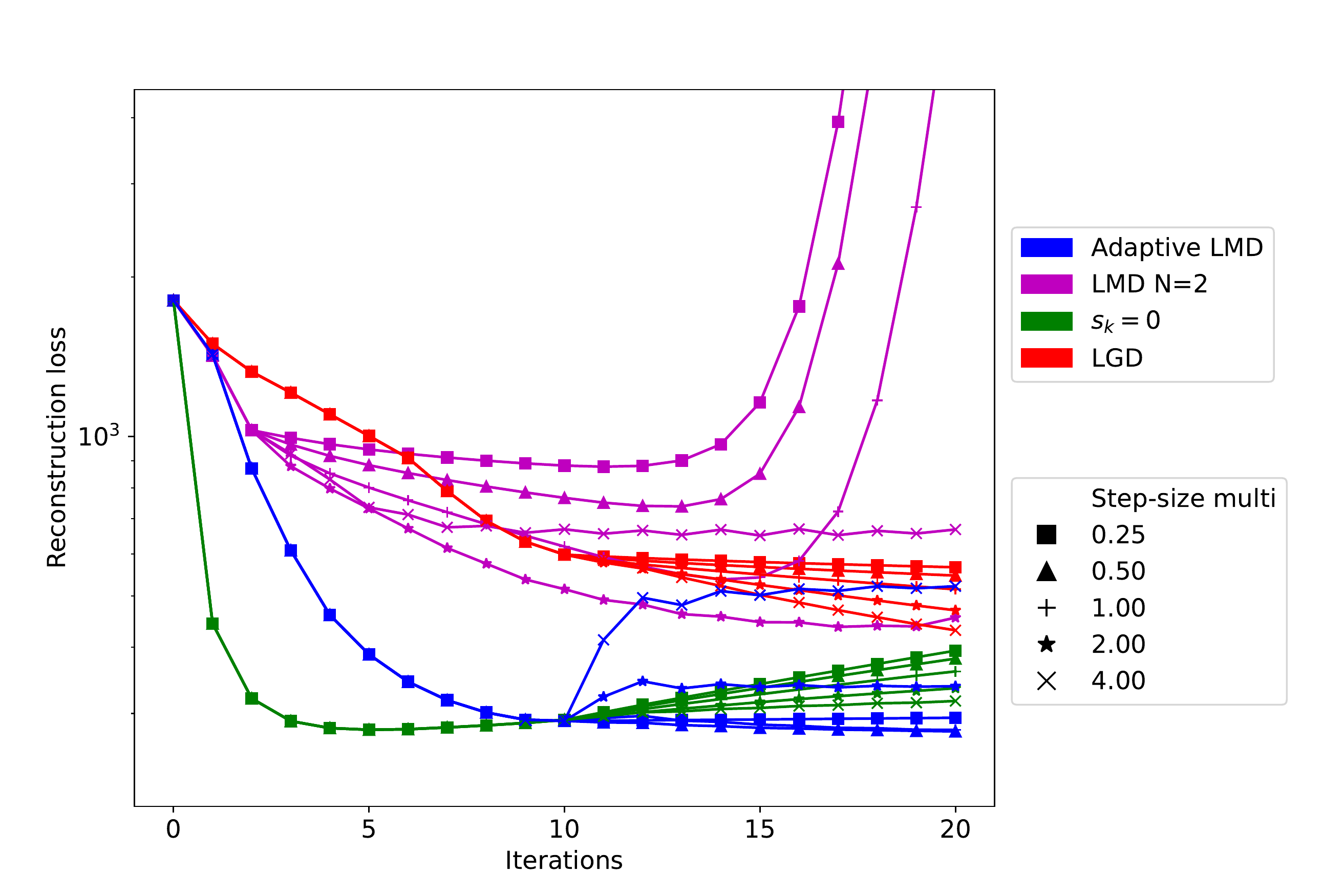}\label{fig:altloss}}}%
    \caption{Ablation study considering the forward-backward loss and reconstruction loss for image inpainting. We consider (i) training a small number of iterations $N=2$, (ii) training without enforcing the forward-backward inconsistency $s_k=0$, and (iii) training where the mirror maps are fixed to be the identity, corresponding to GD with learned step-sizes (LGD). Adaptive LMD is trained for $N=10$ iterations, as in \Cref{sec:inpainting}. Note that LGD does not have a forward-backward loss, as the iterates are exact. }
    \label{fig:ablation}%
\end{figure}%

\subsection{Computational Complexity}
In this subsection, we discuss the computational complexity of the LMD method in terms time and memory, for both training and testing time.


For a single backward and forward pass, the proposed method scales linearly with the dimension and number of iterations. In particular, suppose the space we wish to optimize is over $\mathcal{X}\subseteq \mathbb{R}^d$ with dimension $d$ and batch size $n$, and we run $N$ iterations of LMD. Assuming that backpropagating and taking gradients scales linearly with the number of parameters $P$, the backwards pass takes $\mathcal{O}(n\times d\times N \times P)$ time and $\mathcal{O}(n\times d\times N \times P)$ memory. The forwards pass takes $\mathcal{O}(n\times d\times N \times P)$ time and $\mathcal{O}(n\times d\times P)$ memory, where we drop a factor of $N$ as holding intermediate iterates is not required. 

\Cref{tab:gputime} compares the GPU wall-times and memory consumption for various numbers of training iterations $N$, tested for the STL-10 inpainting experiment for both training and testing. We find that the times and memory consumption are as expected, with near-linear increase in time and train memory, and near-constant test memory.

\begin{table}[h]
\centering
\caption{Table of GPU wall time and memory consumption for training and testing LMD, with various iteration counts. Times are per batch, with a batch-size of 25 on STL-10 images with dimension $3 \times 96 \times 96$. Training and testing was done on Quadro RTX 6000 GPUs with 24GB of memory.}
\label{tab:gputime}
\resizebox{\textwidth}{!}{%
\begin{tabular}{@{}lllll@{}}
\toprule
Iterations & Train time (s) & Test time (s) & Train memory (GB) & Test memory (GB) \\ \midrule
$N=2$  & 5.13  & 0.422 & 8.24  & 1.52 \\
$N=5$  & 8.26  & 1.05  & 12.68 & 1.53 \\
$N=10$ & 13.41 & 2.11  & 20.09 & 1.54 \\ 
$N=20$ & -     & 4.22  & -     & 1.57 \\
$N=50$ & -     & 10.55 & -     & 1.64 \\
\bottomrule
\end{tabular}%
}
\end{table}
\color{black}

\section{Discussion and Conclusions}
\label{sec:conclusions}
%



In this work, we proposed a new paradigm for learning-to-optimize with theoretical convergence guarantees, interpretability, and improved numerical efficiency for convex optimization tasks in data science, based on learning the optimal Bregman distance of mirror descent modeled by input-convex neural networks.  Due to this novel functional parameterization of the mirror map, and by taking a structured and theoretically-principled approach, we are able to provide convergence guarantees akin to the standard theoretical results of classical mirror descent. We then demonstrate the effectiveness of our LMD approach via extensive experiments on various convex optimization tasks in data science, comparing to classical gradient-based optimizers. The provable LMD approach achieves competitive performance with Adam, a heuristically successful method. However, Adam lacks convergence guarantees for the convex case, achieving only local convergence \cite{reddi2019adam, bock2019Adam, zou2018adam}. LMD is able to achieve the fast convergence rates from Adam, while retaining convergence guarantees from slower classical methods such as GD. 

In this paper, we have only considered the most basic form of mirror descent as our starting point. There is still much potential for further improvements on both theoretical results and numerical performance of the algorithm. If a deep parameterization of convex functions with closed form convex conjugate exists, then this would allow for exact convergence. One open question is what an optimal mirror map should look like for a particular problem class such as image denoising, and how well a deep network is able to approximate it. Our ongoing works include accelerating the convergence rates of LMD with momentum acceleration techniques which have been developed for accelerating classical mirror descent \cite{hanzely2021accelerated,krichene2015accelerated}, and stochastic approximation schemes \cite{xu2018accelerated}.

\bibliographystyle{siamplain}
\bibliography{references}
\end{document}


\maketitle

\section[Proof of Lemma]{Proof of \cref{lem:approxMD}}
Recall the approximate mirror descent scheme:
\begin{equation}
    \begin{split}
        &\tilde{x}_1 = x_1,\\
        &x_{k+1} = \arg \min_{x \in X}\left\{\langle x, t_k f'(\tilde{x}_k) \rangle + B(x, \tilde{x}_k)\right\} = \nabla \Psi^* (\nabla\Psi(\tilde{x}_k) - t_k f'(\tilde{x}_k)), \\
        & \tilde{x}_{k+1} \approx x_{k+1}.
    \end{split}
\end{equation}
\begin{proof}
We apply an amortization and splitting technique similar to that in \cite{Duchi10compositeobjective,BECK2003167}. We wish to find an upper bound on the following:
\[t_k f(\tilde{x}_k) - t_k f(x^*) + (B(x^*, \tilde{x}_{k+1}) - B(x^*, \tilde{x}_{k}))\]
Note, since $\nabla \Psi^* = (\nabla \Psi)^{-1}$,
\[\nabla \Psi(x_{k+1}) = \nabla\Psi(\tilde{x}_k) - t_k f'(\tilde{x}_k).\]
We have the following computation:
\begin{align*}
    &\quad B(x^*, \tilde{x}_{k+1}) - B(x^*, \tilde{x}_{k}) \\
    &= \Psi(\tilde{x}_{k}) - \Psi(\tilde{x}_{k+1}) - \langle\nabla \Psi(\tilde{x}_{k+1}) ,x^* - \tilde{x}_{k+1}\rangle + \langle\nabla \Psi(\tilde{x}_{k}) ,x^* - \tilde{x}_{k}\rangle  \\
    &= \Psi(\tilde{x}_{k}) - \Psi(\tilde{x}_{k+1}) - \langle  \nabla\Psi (x_{k+1}), x^* - \tilde{x}_{k+1} \rangle - \langle \nabla \Psi(\tilde{x}_{k+1}) \\
    & \qquad - \nabla \Psi(x_{k+1}), x^* - \tilde{x}_{k+1} \rangle + \langle \nabla \Psi(\tilde{x}_k), x^* - \tilde{x}_k\rangle \\
    &= \Psi(\tilde{x}_{k}) - \Psi(\tilde{x}_{k+1}) - \langle \nabla \Psi(\tilde{x}_k) - t_k f'(\tilde{x}_k), x^* - \tilde{x}_{k+1} \rangle \\
    &\qquad - \langle \nabla \Psi(\tilde{x}_{k+1})  - \nabla \Psi(x_{k+1}), x^* - \tilde{x}_{k+1} \rangle + \langle \nabla \Psi(\tilde{x}_k), x^* - \tilde{x}_k\rangle \\
    &= \Psi(\tilde{x}_k) - \Psi(\tilde{x}_{k+1}) + \langle \nabla \Psi(\tilde{x}_k), \tilde{x}_{k+1} - \tilde{x}_k \rangle \\
    &\qquad + \langle t_k f'(\tilde{x}_k),  x^* - \tilde{x}_{k+1} \rangle \\
    &\qquad - \langle \nabla \Psi(\tilde{x}_{k+1}) - \nabla \Psi(x_{k+1}), x^* - \tilde{x}_{k+1} \rangle.
\end{align*}

The first line in final expression is exactly $-B(\tilde{x}_{k+1}, \tilde{x}_{k})$.  By $\sigma$-strong-convexity of $\Psi$, this is less than $-\sigma \|\tilde{x}_{k+1} - \tilde{x}_{k}\|^2/2$. So we have:
\begin{align*}
    &B(x^*, \tilde{x}_{k+1}) - B(x^*, \tilde{x}_{k}) \\
    &\le -\frac{\sigma}{2} \|\tilde{x}_{k+1} - \tilde{x}_{k}\|^2 + \langle t_k f'(\tilde{x}_k),  x^* - \tilde{x}_{k+1} \rangle - \langle \nabla \Psi(\tilde{x}_{k+1}) - \nabla \Psi(x_{k+1}), x^* - \tilde{x}_{k+1} \rangle
\end{align*}

Returning to the amortization.

\begin{align*}
    &\quad t_k f(\tilde{x}_k) - t_k f(x^*) + (B(x^*, \tilde{x}_{k+1}) - B(x^*, \tilde{x}_{k})) \\
    &\le t_k f(\tilde{x}_k) - t_k f(x^*) + \langle t_k f'(\tilde{x}_k), x^* - \tilde{x}_k\rangle \\
    & \qquad + \langle t_k f'(\tilde{x}_k) , \tilde{x}_k - \tilde{x}_{k+1} \rangle \\
    & \qquad - \frac{\sigma}{2}\|\tilde{x}_{k+1} - \tilde{x}_{k}\|^2 - \langle \nabla \Psi(\tilde{x}_{k+1}) - \nabla \Psi(x_{k+1}), x^* - \tilde{x}_{k+1} \rangle. \\
\end{align*}
\end{proof}

\section{Proof of \cref{thm:approxCOMID,thm:approxCOMIDConstrained}}

\bibliographystyle{siamplain}
\bibliography{references}